\documentclass{amsart}
\usepackage[all]{xy}
\usepackage{enumerate}
\usepackage{subcaption}

\usepackage[utf8]{inputenc}
\usepackage[T1]{fontenc}
\usepackage{microtype}

\usepackage{stix}

\usepackage[unicode,naturalnames,hidelinks]{hyperref}

\DeclareMathOperator{\mypart}{\mathsf{COB}}
\DeclareMathOperator{\mylin}{\mathsf{LCU}}

\newtheorem{theorem}[equation]{Theorem}

\newtheorem{construction}[equation]{Lemma/Construction}
\newtheorem{assumption}[equation]{Assumption}

\newtheorem{lemma}[equation]{Lemma}
\newtheorem{corollary}[equation]{Corollary}

\newtheorem{fact}[equation]{Fact}
\newtheorem{facts}[equation]{Facts}
\theoremstyle{remark}
\newtheorem{remark}[equation]{Remark}
\theoremstyle{definition}
\newtheorem{definition}[equation]{Definition}
\numberwithin{equation}{section}

\newcommand{\on}{\mathord\upharpoonright}
\newcommand{\ve}{\varepsilon}
\newcommand{\forces}{\Vdash}
\newcommand{\Eor}{\mathbb{E}}
\DeclareMathOperator{\dom}{dom}

\newcommand{\textdef}{\emph}

\newcommand{\mye}{*+[F.]{\phantom{\lambda}}}

\DeclareMathOperator{\crit}{cr}

\newcommand{\myparti}{\mypart_i}
\newcommand{\mypartI}{\mypart_1}
\newcommand{\mypartII}{\mypart_2}
\newcommand{\mypartIII}{\mypart_3}
\newcommand{\mypartIV}{\mypart_4}
\newcommand{\mylini}{\mylin_i}
\newcommand{\mylinI}{\mylin_1}
\newcommand{\mylinII}{\mylin_2}
\newcommand{\mylinIII}{\mylin_3}
\newcommand{\mylinIV}{\mylin_4}
\DeclareMathOperator{\cf}{cf}
\DeclareMathOperator{\supp}{supp}
\DeclareMathOperator{\Rel}{R}
\newcommand{\Ri}{\Rel_i}
\DeclareMathOperator{\RI}{R_1}
\DeclareMathOperator{\RII}{R_2}
\DeclareMathOperator{\RIII}{R_3}
\DeclareMathOperator{\RIV}{R_4}
\newcommand{\mR}{\mathbb{R}}

\newcommand{\Pa}{\mathbb{P}^5}
\newcommand{\PaB}{\mathbb{P}^{*,5}}
\newcommand{\Pai}{\mathbb{P}^i}

\newcommand{\Paell}{\mathbb{P}^\ell}
\newcommand{\Paip}{\mathbb{P}^{i+1}}

\newcommand{\PaIX}{\mathbb{P}^9}

\DeclareMathOperator{\cov}{cov}
\DeclareMathOperator{\cof}{cof}
\DeclareMathOperator{\non}{non}
\DeclareMathOperator{\add}{add}

\newcommand{\covnull}{\cov(\mathcal N)}
\newcommand{\cofnull}{\cof(\mathcal N)}
\newcommand{\addnull}{\add(\mathcal N)}
\newcommand{\nonnull}{\non(\mathcal N)}
\newcommand{\covmeager}{\cov(\mathcal M)}
\newcommand{\cofmeager}{\cof(\mathcal M)}
\newcommand{\addmeager}{\add(\mathcal M)}
\newcommand{\nonmeager}{\non(\mathcal M)}
\newcommand{\BUP}[1]{#1^U}

\date{\today}
\title{Cicho\'{n}'s maximum}
\thanks{Supported by 
Austrian Science Fund (FWF): P29575 \& I3081 
and 
National Science Foundation NSF DMS-1362974 (first author),
Austrian Science Fund (FWF): P26737 \& P30666 (second author)
European Research Council grant ERC-2013-ADG 338821 (third author).
This is publication number 1122 of the third author.}
\author{Martin Goldstern}
\address{Institute of Discrete Mathematics and Geometry\\
Technische Universität Wien (TU Wien).}
\email{martin.goldstern@tuwien.ac.at}
\urladdr{http://www.tuwien.ac.at/goldstern/}
\author{Jakob Kellner}
\address{Institute of Discrete Mathematics and Geometry\\
Technische Universität Wien (TU Wien).}
\email{jakob.kellner@tuwien.ac.at}
\urladdr{http://dmg.tuwien.ac.at/kellner/}
\author{Saharon Shelah}
\address{The Hebrew University of Jerusalem and Rutgers University.}
\email{shlhetal@mat.huji.ac.il}
\urladdr{http://shelah.logic.at/}
\begin{document}

\begin{abstract}
Assuming  four strongly compact cardinals, it is consistent that all entries in Cicho\'n's  diagram (apart from $\addmeager$ and $\cofmeager$, whose values are determined by the others) 
are pairwise different; more
specifically that
\[
\aleph_1 < \addnull < \covnull < \mathfrak{b} <
\nonmeager < \covmeager <
 \mathfrak{d} < \nonnull < \cofnull < 2^{\aleph_0}.\]
\end{abstract}

\maketitle
\section*{Introduction}
\subsection*{Independence}
How many Lebesgue null sets are required to
cover the real line?\\
Obviously countably many
are not enough,
as the countable union of null sets is null;
and obviously continuum many are enough,
as $\bigcup_{r\in\mathbb{R}}\{r\}=\mathbb{R}$.

The answer to our question is a cardinal number called
$\covnull$. As we have just seen,
\[
\aleph_0=|\mathbb{N}|<\covnull\le |\mathbb{R}|=2^{\aleph_0}.
\]
In particular, 
if the Continuum Hypothesis (CH) holds (i.e., if there are 
no cardinalities strictly between $|\mathbb{N}|$ and 
$|\mathbb{R}|$, or equivalently: if $\aleph_1=2^{\aleph_0}$), then 
$\covnull=2^{\aleph_0}$; but 
without CH, the answer could also be some cardinal
less than $2^{\aleph_0}$.
According to Cohen's famous result~\cite{MR0157890}, 
CH is independent of the usual axiomatization 
of mathematics, the set theoretic axiom system ZFC. I.e.,
we can prove that the ZFC axioms neither imply CH nor imply $\lnot$CH.
For this result, Cohen introduced the method of forcing,
which has been continuously expanded and refined ever since.
Forcing also proves that the
value of $\covnull$ is independent. For example,
$\covnull=\aleph_1<2^{\aleph_0}$ is consistent, as is 
$\aleph_1<\covnull=2^{\aleph_0}$.

\subsection*{Cicho\'n's  diagram}
$\covnull$ is a so-called cardinal characteristic of the continuum. Other
well-studied
characteristics include the following:%
\begin{itemize}
\item
$\addnull$ is the smallest number of Lebesgue null sets whose
union is not null.
\item $\nonnull$ is the smallest cardinality
of a non-null set.
\item $\cofnull$
is the smallest size of a cofinal family of null sets,
i.e., a family that contains for each null set $N$ 
a superset of $N$.
\item 
Replacing ``null'' with ``meager'', we can 
analogously define $\addmeager$,
$\nonmeager$, $\covmeager$, and $\cofmeager$.
\item 
In addition, we define $\mathfrak{b}$ as the smallest size of an
unbounded family, i.e., a family $\mathcal{H}$ of functions
from $\mathbb{N}$ to $\mathbb{N}$ such that for every 
$f:\mathbb{N}\to\mathbb{N}$ there is some $h\in\mathcal{H}$
which is not almost everywhere bounded by $f$.

Equivalently, $\mathfrak{b}=\add(\mathcal K)=\non(\mathcal K)$,
where $\mathcal K$ is the $\sigma$-ideal generated by the compact subsets of the irrationals.
\item
And $\mathfrak{d}$ is the smallest size of a
dominating family, i.e., a family $\mathcal{H}$ such that for every 
$f:\mathbb{N}\to\mathbb{N}$ there is some $h\in\mathcal{H}$
such that $(\exists n\in\mathbb{N})\,(\forall m>n)\, h(m)>f(m)$.

Equivalently, $\mathfrak{d}=\cov(\mathcal K)=\cof(\mathcal K)$.
\item For the ideal $\text{ctbl}$ of countable sets, we trivially get 
$\add(\text{ctbl})=\non(\text{ctbl})=\aleph_1$ and $\cov(\text{ctbl})=\cof(\text{ctbl})=2^{\aleph_0}$.
\end{itemize}

The characteristics
 we have mentioned so far,\footnote{There are many other cardinal characteristics, see for example~\cite{MR2768685}, but the ones in Cicho\'n's  diagram
 seem to be considered to be the most important ones.} and the basic relations between them, can be summarized 
in Cicho\'n's  diagram:
\[
\xymatrix@=2.5ex{
           & \covnull\ar[r]        & \nonmeager \ar[r]      &  \cofmeager \ar[r]     & \cofnull\ar[r]  &2^{\aleph_0} \\
           &                    & \mathfrak b\ar[r]\ar[u]  &  \mathfrak d\ar[u] &              &\\ 
\aleph_1\ar[r] & \addnull\ar[r]\ar[uu] & \addmeager\ar[r]\ar[u] &  \covmeager\ar[r]\ar[u]& \nonnull\ar[uu] &
}
\]    
An arrow from $\mathfrak x$ to $\mathfrak y$ indicates that ZFC proves
$\mathfrak x\le \mathfrak y$. 
Moreover, $\max(\mathfrak d,\nonmeager)=\cofmeager$ and $\min(\mathfrak b,\covmeager)=\addmeager$. 
A (by now) classical series of
theorems~\cite{MR719666,MR1233917,MR781072,MR1071305,MR1022984,MR613787,MR735576,MR697963,MR800191}
proves these (in)equalities in ZFC and shows that they are the only ones
provable. More precisely, all assignments of the values $\aleph_1$ and $\aleph_2$
to the characteristics in Cicho\'n's Diagram are consistent with ZFC, provided they do
not contradict the above (in)equalities.  (A complete proof can be found
in~\cite[ch.~7]{BJ}.)

Note that Cicho\'n's  diagram shows a fundamental asymmetry
between the ideals of Lebesgue null sets and of meager sets
(we will mention another one in the context of large cardinals).
Any such asymmetry is hidden if we assume CH, as under CH
not only all the characteristics are $\aleph_1$, but even
the Erd\H os-Sierpi\'nski Duality Theorem holds~\cite[ch.\ 19]{MR584443}:
There is an involution $f:\mathbb{R}\to \mathbb{R}$ (i.e., a bijection such that $f\circ f=\text{Id}$) such that
$A\subseteq \mathbb{R}$ is meager  iff
$f''A$ is null.

So it is settled which assignments of $\aleph_1$ and $\aleph_2$ to
Cicho\'n's  diagram are consistent. 
It is more challenging to show that the diagram can 
contain more than two different cardinal values. 
For recent progress in this direction
see, e.g., \cite{MR3047455,MR3513558,five,ourselves}.

The result of this paper is in some respect the strongest possible, 
as we show that consistently \emph{all} the entries are pairwise different (apart 
from the two equalities provable in ZFC mentioned above). Of course one can ask more; see the questions in Section~\ref{sec:questions}. In particular, we 
use large cardinals in the proof.

\subsection*{Large cardinals}
As mentioned, ZFC is an axiom system for the whole of mathematics.
A much ``weaker'' axiom system (for the natural numbers) is PA
(Peano arithmetic).

Gödel's Incompleteness Theorem shows that a theory
such as PA or ZFC can never prove its own consistency. 
On the other hand, it is trivial to show in ZFC that PA is consistent (as in
ZFC we can construct $\mathbb{N}$ and prove that it satisfies PA). We can say that ZFC has a higher 
consistency strength than PA.

One axiom of ZFC is INF, the statement ``there is an infinite cardinal''. If we remove INF from ZFC, we end up with a theory ZFC$^0$ that can still describe concrete hereditarily finite objects and can be interpreted (admittedly in a not very natural way) as a weak version of PA which has the same consistency strength as PA.\footnote{More concretely, ZF$_\text{fin}\coloneq \text{ZFC}^0+\lnot\text{INF}$ can be seen to be ``equivalent'' to PA (i.e., mutually interpretable); this goes back to Ackermann~\cite{MR1513141}, see the survey~\cite{MR2357524}.}
So we can say that adding an infinite cardinal to ZFC$^0$ increases the consistency strength.

There are notions of cardinals numbers much ``stronger'' than just ``infinite''.
Often, such large cardinal assumptions (abbreviated LC in the following) have the following form:
\begin{quote}
There is a cardinal $\kappa>\aleph_0$ that behaves towards the smaller cardinals 
in a similar way as $\aleph_0$ behaves to finite numbers.
\end{quote}
A forcing proof shows, e.g.,
\begin{quote}
If ZFC is consistent, then ZFC+$\lnot$CH is consistent,
\end{quote}
and this implication can be proved in a very weak system such as PA.
However, we cannot prove (not even in ZFC) for any large cardinal
\begin{quote}
``if ZFC is consistent, then ZFC+LC is consistent''; 
\end{quote}
because in ZFC+LC we can prove the consistency of ZFC.
We say: LC has a higher consistency strength than ZFC.

An instance of a large cardinal (in fact a very weak one, a so-called inaccessible cardinal),
appears in another striking example of the asymmetry between measure and category:
The following statement is equiconsistent with an inaccessible cardinal~\cite{MR0265151,MR768264}:
\begin{quote}
All projective\footnote{This is the smallest family containing the Borel sets and closed under continuous images, complements, and countable unions. In practice, all
sets used in mathematics that are defined without using AC are projective.
Alternatively we could use the statement: ``ZF (without the Axiom of choice) holds and all sets of reals are Lebesgue measurable.''} sets of reals are 
Lebesgue measurable.
\end{quote}
In contrast, according to~\cite{MR768264} no large cardinal assumption is required
to show the consistency of
\begin{quote}
All projective sets of reals have the property of Baire.
\end{quote}
So we can assume ``for free'' that all (reasonable) sets have the Baire property,
whereas we have to provide additional
consistency strength for Lebesgue measurability.

In the case of our paper, we require (the consistency of)
the existence of four compact cardinals to prove our main result.
It seems unlikely that any large cardinals
are actually required; but a proof without them would probably be considerably more complicated. It is not unheard of that ZFC results first have (simpler) proofs
using large cardinal assumptions; an example can be found in~\cite{MR2096454}.

\subsection*{Annotated Contents}

{}  From now on, we 
assume that the reader is familiar with some basic
properties 
of the characteristics defined above, as well as with the
associated 
forcing notions Cohen, amoeba, random, Hechler and eventually different, all of
which can  be found, e.g., in~\cite{BJ}.


This paper consists of three parts:

In Section~\ref{sec:partA},
we present a finite support ccc iteration $\Pa$ forcing that
$\aleph_1 < \addnull < \covnull < \mathfrak{b}<\nonmeager<\covmeager=2^{\aleph_0}$. 
This result is not new: Such a forcing was introduced in~\cite{MR3513558}, and we follow this construction  
quite closely. However, we need GCH
in the ground model, whereas \cite{MR3513558}
requires $2^\chi\gg \lambda$ for some $\chi<\lambda$.
Also, we describe how the inequalities are ``strongly witnessed'',
see Definitions~\ref{def:linear} and~\ref{def:partial}. 

In Section~\ref{sec:partB}, we show how to construct (under GCH) for $\kappa$ strongly compact and $\theta>\kappa$ regular a ``BUP-embedding'' from $\kappa$ to $\theta$, i.e., an elementary
embedding $j:V\to M$ with critical point $ \kappa$ and $\cf(j(\kappa))=|j(\kappa)|=\theta$
such that $M$ is transitive and ${<}\kappa$-closed 
and such that $j''S$ is cofinal in $j(S)$ for 
every $\le\kappa$-directed partial order $S$.
For a ccc forcing $P$ we investigate $j(P)$
and show that $j(P)$ forces the same values to some
characteristics in Cicho\'n's diagram as $P$
and different values to others, in a very controlled way; assuming that there were
``strong witnesses'' for $P$ forcing the inital values, as described in Section~\ref{sec:partA}.

Section~\ref{sec:partC} 
contains the main result of this paper: Assuming four strongly compact cardinals, we let $k$ be the
composition of four such BUP-embeddings, mapping
$\Pa$ to a ccc forcing $\PaIX$. We  
then show that $\PaIX$ forces
\[
\aleph_1 < \addnull < \covnull < \mathfrak{b} <\nonmeager<\covmeager< \mathfrak{d} < \nonnull < \cofnull < 2^{\aleph_0},\]
i.e., we get for increasing cardinals $\lambda_i$ the constellation of Figure~\ref{fig:ourorder}.
\begin{figure}
  \centering
\[
\xymatrix@=2.5ex{
           & \lambda_2\ar[r]        & \lambda_4 \ar[r]      &  \mye \ar[r]     & \lambda_8\ar[r]  &\lambda_9 \\
           &                    & \lambda_3\ar[r]\ar[u]  &  \lambda_6\ar[u] &              &\\ 
\aleph_1\ar[r] & \lambda_1\ar[r]\ar[uu] & \mye\ar[r]\ar[u] &  \lambda_5\ar[r]\ar[u]& \lambda_7\ar[uu] &
}
\]    
    \caption{\label{fig:ourorder}Our cardinal configuration (the $\lambda_i$ are increasing).}
\end{figure}
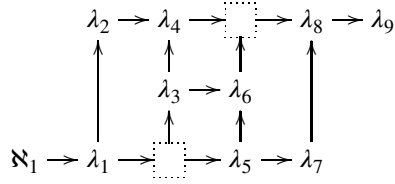

Boolean ultrapowers as used in this paper  were 
investigated by Mansfield~\cite{MR0300887}
and recently applied e.g.\ by  the third author with Malliaris~\cite{MR3451934} and with Raghavan~\cite{F1611}, where
Boolean ultrapowers of forcing notions are used to force specific values to certain cardinal characteristics.
Recently the third author developed a method of using Boolean ultrapowers to control characteristics in Cicho\'n's diagram. A first
(and simpler) application of these methods is given in~\cite{ourselves}.

We mention some open questions in Section~\ref{sec:questions}.

\subsection*{Acknowledgments}
We thank three anonymous referees for pointing out several
unclarities and typos; and Moti Gitik and Diego Mej\'ia
for suggestions to improve the presentation.

\section{The initial forcing}\label{sec:partA}

\subsection{Good iterations and the \texorpdfstring{$\mylin$}{LCU} property}
We want to show that some forcing $\Pa$
results in $\mathfrak{x}=\lambda_i$ for certain characteristics 
$\mathfrak{x}$.
So we have to show two ``directions'', 
$\mathfrak{x}\le \lambda_i$ 
and $\mathfrak{x}\ge \lambda_i$. For most of the characteristics,
one direction will 
 use the fact that $\Pa$ is ``good'';
a notion
introduced by Judah and the third author~\cite{MR1071305} and Brendle~\cite{MR1129144}.
We now recall the basic facts of good iterations, and specify the instances 
of the relations we use.

\begin{assumption}\label{asm:hwtt}
We will consider binary relations $\Rel$ on $X=\omega^\omega$ (or on $X=2^\omega)$ that satisfy the following: There are relations $\Rel^k$ such that
$\Rel=\bigcup_{k\in\omega} \Rel^k$, each $\Rel^k$ is a closed subset (and in fact absolutely defined)
of $X\times X$, and for  $g\in X$
and $k\in\omega$, the set $\{f\in X:\, f\Rel^k g\}$ is nowhere dense (and of course closed). 
Also, for all $g\in X$ there is some $f\in X$ with $f \Rel g$.
\end{assumption} 

We will actually use another space as well,
the space $\mathcal C$ 
of strictly positive rational sequences $(q_n)_{n\in\omega}$ such that
$\sum_{n\in\omega}q_n\le 1$.
It is easy to see that $\mathcal C$ is homeomorphic to
$\omega^\omega$, when
we equip the rationals with the discrete topology 
and use the product topology. Let us fix one such (absolutely defined) homeomorphism.

We use the following instances of relations $\Rel$ 
on $X$; it is easy
to see that they all satisfy the assumption (for $X_1=\mathcal C$ we use the homeomorphism mentioned above):
\begin{definition}
\begin{itemize}
\item[1.] $X_1=\mathcal C$: $f \RI g$  if
$(\forall^*n\in\omega)\, f(n)\le g(n)$.\\
(We use $\forall^*n$ as abbreviation for $(\exists n_0)\, (\forall n>n_0)$.)
\item[2.] 
Fix a partition 
$(I_n)_{n\in\omega}$ of $\omega$ with $|I_n|=2^{n+1}$.
\\
$X_2=2^\omega$: $f \RII g$ if $(\forall^* n\in\omega)\, f\restriction I_n\neq g\restriction I_n$.
\item[3.] $X_3= \omega^\omega$: $f \RIII g$  if $(\forall^* n\in\omega)\, f(n)\le g(n)$.
\item[4.] $X_4= \omega^\omega$: $f \RIV g$  if $(\forall^* n\in\omega)\, f(n)\neq g(n)$.
\end{itemize}
\end{definition}

Note that Assumption~\ref{asm:hwtt} is satisfied, witnessed by the relations $\Rel^k_i$ 
defined by replacing $(\forall^* n\in\omega)$ with $(\forall n\ge k)$.

We say ``$f$ is bounded by $g$'' if $f\Rel g$; and, for $\mathcal Y\subseteq \omega^\omega$, 
``$f$ is bounded by $\mathcal Y$'' if
$(\exists y\in \mathcal Y)\, f \Rel y$. We say ``unbounded'' for ``not bounded''. (I.e., $f$ is unbounded by $\mathcal Y$ if $(\forall y\in \mathcal Y)\,\lnot f\Rel y$.) 
We call $\mathcal X$  an $\Rel$-unbounded family, if $\lnot (\exists g)\,(\forall x\in \mathcal X)x\Rel g$, and an $\Rel$-dominating family if $(\forall f)\,(\exists x\in \mathcal X)\, f\Rel x$. 
\begin{itemize}
   \item
Let $\mathfrak{b}_i$ be the minimal size of an $\Rel_i$-unbounded family,
   \item and let  $\mathfrak{d}_i$ be the minimal size
      of an $\Rel_i$-dominating family.
\end{itemize}

We only need the following connections between $\Ri$ and the cardinal characteristics: 
\begin{lemma}\label{lem:connection}
\begin{itemize}
\item[1.]
$\addnull=\mathfrak{b}_1$ and $\cofnull=\mathfrak{d}_1$.
\item[2.]
$\covnull\le\mathfrak{b}_2$ and $\nonnull\ge\mathfrak{d}_2$.
\item[3.]
$\mathfrak{b}=\mathfrak{b}_3$ and $\mathfrak{d}=\mathfrak{d}_3$.
\item[4.]
$\nonmeager=\mathfrak{b}_4$ and $\covmeager=\mathfrak{d}_4$.
\end{itemize}
\end{lemma}
\begin{proof}
(3) holds by definition. (1) can be found in~\cite[6.5.B]{BJ}.
(4) is a result of~\cite{MR671224,MR917147},  
cf.~\cite[2.4.1~and~2.4.7]{BJ}.

To prove (2), note that 
for fixed $f\in 2^\omega$ the set $\{g\in 2^\omega: \lnot f\RII g\} $
is a null set, call it $N_f$. 
Let $\mathcal F$ be an $\RII$-unbounded family.
Then $\{N_f:\, f\in\mathcal{F}\}$ covers $2^\omega$:
Fix $g\in 2^\omega$.
As $g$ does not bound $\mathcal{F}$, there is some 
$f\in \mathcal{F}$ unbounded by $g$, i.e., $g\in N_f$.
Let $X$ be a non-null set. 
Then $X$ is $\RII$-dominating: For any $f\in 2^\omega$
there is some $x\in X\setminus N_f$, i.e., $f \RII x$.
\end{proof}

We will also use:
\begin{lemma}\label{lem:connection2}\cite{BJ}
Amoeba forcing $\mathbb A$ adds a dominating element $\bar b$ of $\mathcal{C}$,
i.e., $\mathbb A\forces \bar q\RI \bar b$ for all $\bar q\in \mathcal{C}\cap V$.
\end{lemma}

\begin{proof}
Let us define a slalom $\mathcal S$ to be a function
$\mathcal S:\omega\to [\omega]^{{<}\omega}$ such that 
$|\mathcal S(n)|>0$ and 
$\sum_{n=1}^{\infty} \frac{|\mathcal S(n)|}{n^2}<\infty$.

Amoeba forcing will add a null set covering all
old null sets, and therefore (according to \cite[2.3.3]{BJ}) a slalom $\mathcal S$ covering all old slaloms. 
Set 
$a_n \coloneq \frac{|\mathcal S(n)|}{n^2}$, 
$M \coloneq \sum_{n=1}^\infty a_n$, set $M'$ the smallest natural number $\ge M$, 
and set $b_n \coloneq \frac{a_{n+1}}{M'}$. Then
it is easy to see that
$(b_n)_{n\in\omega}\in\mathcal C$  dominates every old  sequence $(q_n)_{n\in\omega}$ in~$\mathcal C$.
\end{proof}

\begin{definition}\cite{MR1071305}
Let $P$ be a ccc forcing, $\lambda$ an uncountable regular cardinal, and $\Rel$ as above.
$P$ is $(\Rel,\lambda)$-good,
if for each $P$-name $r\in\omega^\omega$ there is (in $V$) a nonempty
set $\mathcal Y\subseteq \omega^\omega$ of size ${<}\lambda$
such that every $f$ (in $V$) that is $\Rel$-unbounded by $\mathcal Y$ is forced to be $\Rel$-unbounded by $r$ as well.
\end{definition}

Note that $\lambda$-good trivially implies $\mu$-good if $\mu\ge\lambda$
are regular.

How do we get good forcings? Let us just note the following results:
\begin{lemma}\label{lem:gettinggood}
A finite support (henceforth abbreviated FS) iteration of Cohen forcing is good for any $(\Rel,\lambda)$, and
the composition of two $(\Rel,\lambda)$-good forcings is $(\Rel,\lambda)$-good.
\\
Assume that $(P_\alpha,Q_\alpha)_{\alpha<\delta}$ is a FS ccc iteration.
Then $P_\delta$ is $(\Rel,\lambda)$-good, if each $Q_\alpha$ is forced to satisfy the following:
\begin{itemize}
\item[1.] For $\Rel=\RI$:  $|Q_\alpha|<\lambda$, or $Q_\alpha$ is $\sigma$-centered, or  $Q_\alpha$ is a sub-Boolean-algebra of the random algebra.
\item[2.] For $\Rel=\RII$: $|Q_\alpha|<\lambda$, or $Q_\alpha$ is $\sigma$-centered.
\item[4.] For $\Rel=\RIV$: $|Q_\alpha|<\lambda$.
\end{itemize}
\end{lemma}
(Remark: For $\RIII$ the same holds as for $\RIV$, which 
however is of no use for our construction.)

\begin{proof}
$(\Rel,\lambda)$-goodness is preserved by FS ccc iterations (in particular compositions),
as proved in~\cite{MR1071305}, cf.~\cite[6.4.11--12]{BJ}. 
Also, ccc forcings of size ${<}\lambda$ are $(\Rel,\lambda)$-good~\cite[6.4.7]{BJ};
which takes care of the case $|Q_\alpha|<\lambda$ (and in particular of  Cohen forcing).
So it remains to show that (for $i=1,2$) the ``large'' iterands in the list are
$(\Rel_i,\lambda)$-good. 

For $\RI$ 
this follows from~\cite{MR1071305} and~\cite{MR1022984}, cf.~\cite[6.5.17--18]{BJ}.

For $\RII$, this is proven in~\cite{MR1129144}, and as the proof is very short, we give it here:
Write $Q_\alpha$ as union $\bigcup_{k\in\omega}Q^k$ of centered sets.
Given the $Q_\alpha$-name $r$, pick a countable elementary submodel $N$ containing $r$ and $Q_\alpha$, and 
set $\mathcal Y=N\cap 2^\omega$. Assume towards a contradiction that $f$ is unbounded by
$\mathcal Y$, but is forced by $p_0$ to be bounded by $r$, i.e., $p_0$ forces $(\forall n>n_0)f\restriction I_n\ne r\restriction I_n$.
Now $p_0$ may not be in $N$, but there is some $k_0\in \omega$ such that 
$p_0\in Q^{k_0}$.
In $N$, we can pick for each $n\in\omega$ some $s_n\in 2^{I_n}$ such that no
$q\in Q^{k_0}$ forces $r\restriction I_n\ne s_n$. (There are only finitely many $s\in 2^{I_n}$;
if each $s$ is forbidden by some $q$, then the common stronger element would prevent
all possibilities for $r\restriction I_n$.) So in $N$, we get some $g\in 2^\omega$ such that
$g\restriction I_n=s_n$. As $f$ is unbounded by $\mathcal Y$ (or equivalently: by $N$), there is some $n>n_0$ such that
$f\restriction I_n=g\restriction I_n=s_n$, which implies that $p_0$ (as an element of $Q^{k_0}$)
does not force $r\restriction I_n \ne f\restriction I_n$, a contradiction.
\end{proof}

\begin{lemma}\label{lem:coboundedunbounded}
Let $\lambda\le\kappa\le\mu$ be uncountable regular cardinals.
After forcing with $\mu$ many Cohen reals $(c_\alpha)_{\alpha\in \mu}$, followed by 
an $(\Rel,\lambda)$-good forcing, we get: For every real $r$ in the final extension, the set 
$\{\alpha\in \kappa:\, c_\alpha\text{ is unbounded by }r\}$ is cobounded in $\kappa$. I.e., 
$(\exists \alpha\in\kappa)\, (\forall \beta\in \kappa\setminus \alpha)\, \lnot c_\beta\Rel r$.
\end{lemma}

(The 
Cohen real $c_\beta$ can
be interpreted both as Cohen generic element of  $2^\omega$ 
and as Cohen generic element of  $\omega^\omega$; we 
use the interpretation suitable for the relation $\Rel$.)

\begin{proof}
Work in the intermediate extension after $\kappa$ many Cohen reals; let us call it $V_\kappa$.
The remaining forcing (i.e., $\mu\setminus\kappa$ many Cohens composed with the good forcing) is good; so applying the definition we get 
(in $V_\kappa$)
a set $\mathcal{Y}$ of size ${<}\lambda$. 

As the initial Cohen extension is ccc, and $\kappa\ge \lambda$ is regular,
we get some $\alpha\in\kappa$ such that each element $y$ of $\mathcal{Y}$
already exists in the extension by the first $\alpha$ many Cohens, call it
$V_{\alpha}$.
The set of reals $M_y$ bounded by $y$ is meager (and absolute).
Any $c_\beta$ for $\beta\in\kappa\setminus \alpha$
is Cohen over $V_{\alpha}$, and therefore not in $M_y$, i.e., not bounded by $y$,
i.e., not by $\mathcal Y$.
So according to the definition of good, each such $c_\beta$ is 
unbounded by $r$ as well, for the given $r$.
\end{proof}

In the light of this result, let us revisit Lemma~\ref{lem:connection} 
with some new notation, the ``linearly cofinally unbounded'' property $\mylin$:
\begin{definition}\label{def:linear}
For $i=1,2,3,4$, $\gamma$ a limit ordinal,  and $P$ a ccc forcing notion, let $\mylini(P,\gamma)$  stand for:
\begin{quote}
There is a sequence 
$(x_\alpha)_{\alpha\in\gamma}$ of $P$-names of elements of $X_i$ (the domain of the relation $\Ri$) such that 
for every such $P$-name $y$\\
$(\exists \alpha\in\gamma)\, (\forall \beta\in \gamma\setminus \alpha)\,P\forces \lnot x_\beta \Ri y$.
\end{quote}
\end{definition}

\begin{lemma}\label{lem:linearcharacteristics}
\begin{itemize}
\item 
$\mylini(P,\delta)$
is equivalent 
to $\mylini(P,\cf(\delta))$.
\item
If  $\lambda$ is regular, then
$\mylini(P,\lambda)$ 
implies $\mathfrak{b}_i\le\lambda$ and $\mathfrak{d}_i\ge\lambda$. 
\end{itemize}
In particular: 
\begin{itemize}
\item[1.]
$\mylinI(P,\lambda)$ implies  $P\forces(\,\addnull\le\lambda\,\&\,\cofnull\ge\lambda\,)$.
\item[2.] 
$\mylinII(P,\lambda)$ implies  $P\forces(\,\covnull\le\lambda\,\&\,\nonnull\ge\lambda\,)$.
\item[3.] 
$\mylinIII(P,\lambda)$ implies  $P\forces(\,\mathfrak{b}\le\lambda\,\&\,\mathfrak{d}\ge\lambda\,)$.
\item[4.] 
$\mylinIV(P,\lambda)$ implies  $P\forces(\,\nonmeager\le\lambda\,\&\,\covmeager\ge\lambda\,)$.

\end{itemize}
\end{lemma}

\begin{proof} Assume that $(\alpha_\beta)_{\beta\in\cf(\delta)}$ is increasing continuous and cofinal
in $\delta$.
If $(x_\alpha)_{\alpha\in \delta}$
witnesses $\mylini(P,\delta)$,  then $(x_{\alpha_\beta})_{\beta\in\cf(\delta)}$ witnesses 
$\mylini(P,\cf(\delta))$. And 
if $(x_\beta)_{\beta\in\cf(\delta)}$ witnesses 
$\mylini(P,\cf(\delta))$, then  $(y_\alpha)_{\alpha\in\delta}$
witnesses $\mylini(P,\delta)$, where 
$y_{\alpha}\coloneq x_\beta$ for $\alpha\in [\alpha_\beta,\alpha_{\beta+1})$.

The set $\{x_\alpha:\, \alpha\in\lambda\}$ is certainly forced to be $\Rel_i$-unbounded;
and given a set $Y=\{y_j:\, j<\theta\}$ of $\theta<\lambda$ many $P$-names,
each has a bound $\alpha_j\in\lambda$ so that 
$(\forall \beta\in \lambda\setminus \alpha_j)\,P\forces \lnot x_\beta \Ri y_j$, so for any $\beta\in\lambda$ above all $\alpha_j$ we get
$P\forces \lnot x_\beta \Ri y_j$ for all $j$; i.e., $Y$ cannot be dominating.
\end{proof}

\begin{remark} 
$\mathfrak b_i\le \lambda $ is equivalent to the existence of a sequence 
 $(x_\alpha:\, \alpha\in\lambda)$ with the property 
 $(\forall y)\,( \exists \alpha)\, \lnot\bigl( x_\alpha R_i y\bigr)$; such a sequence
 might be called a ```witness'' for $\mathfrak b_i\le \lambda $.  In $\mylin$ we demand a stronger property;  a sequence $(x_\alpha:\alpha<\lambda)$ with this stronger property 
 could informally be called a ``strong witness'' for $\mathfrak b_i\le \lambda $. Similarly, the next subsection introduces a different
 notion, $\mypart$, corresponding to ``strong witnesses'' for 
 $\mathfrak d_i\le \mu$.
\end{remark}

\subsection{The inital forcing \texorpdfstring{$\Pa$}{P5}: Partial forcings and the \texorpdfstring{$\mypart$}{COB} property}\label{ss:initalforcing}

Assume we have a forcing iteration $(P_\beta,Q_\beta)_{\beta<\alpha}$ with limit $P_\alpha$, where each $Q_\beta$ is forced by $P_\beta$ to be a set of reals such that
the generic filter of $Q_\beta$ is determined (in a Borel way)\footnote{More specifically, we require that the Borel function for $Q_\beta$ is already fixed in the ground model.
For example, assume $Q_\beta$ is random forcing, defined as the set of all  positive pruned trees $T$,
i.e., trees $T\subseteq 2^{{<}\omega}$ without leaves such that $[T]$ has positive measure. Then the generic filter $G$ for this forcing
is determined by the generic real $\eta$ (the random real),
and $G$ consists of those trees $T$ such that $\eta\in [T]$, which is a Borel relation.
See~\cite[Sec.~1.2]{ourselves} for a formal definition and more details.} from some generic real $\eta_\beta$.
Fix some $w\subseteq \alpha$. We define the $P_\alpha$-name 
$Q_\alpha$ to consist of all random forcing conditions that can be Borel-calculated from generics at $w$ alone.

More explicitly:
\begin{definition}\label{def:groundmodelsequence}
\begin{enumerate}
    \item $q$ is in $Q_\alpha$ if there are in the ground model $V$ 
a countable subset $u\subseteq w$ and 
a Borel function $B:\mR^ u\to \mR $
such that 
$q=B(\, (\eta_\beta)_{\beta\in u} \, ) $ is a random condition.

Being a random condition is a Borel property (if
we fix some suitable representation of random forcing).
Accordingly, we can restrict ourselves to the 
case that $B$ is a Borel function whose image 
consists of random conditions only.
\item We call a pair $(B,u)$ as above ``a $w$-groundmodel-code'' or just ``code''.
Note that this code is a ground model object.
$Q_\alpha$ consists exactly of the evaluations of such codes.
\item\label{item:det}
We call a condition $(p,q)\in P_\alpha*Q_\alpha$
``determined at position $\alpha$'',
if there is a code $(B,u)$ such that $p$ forces that
$(B,u)$ is a code for $q$.
(Note that generally we only have  
a $P_\alpha$-name for a code.)
Given some $(p,q)$, we can obviously find $p'\le p$
such that $(p',q)$ is determined at $\alpha$.
\item\label{item:gms} 
We will later also consider 
so-called ``groundmodel-code-sequences'' 
for elements of $Q_\alpha$, that is (in $V$) a sequence $(B_n,u_n)_{n\in\omega}$  of codes,
where  $u_n$ is in $w_\alpha$.
Of course not every $\omega$-sequence 
of $Q_\alpha$-conditions in the $P_\alpha$-extension is described by a ground model sequence. (In particular, there will only be few ground model sequences, but many
new $\omega$-sequences in the extension.)
\end{enumerate}
\end{definition}
Clearly, in the $P_\alpha$ extension, $Q_\alpha$ is a subforcing (not necessarily a complete one) of the full random forcing,
and if $p,q$ in $Q_\alpha$ are incompatible in $Q_\alpha$
then they are incompatible in random forcing. (Two 
compatible conditions $p,q$ have a canonical conjunction $p\wedge q$ (the intersection),
and if $p$ and $q$ are both Borel-calculated from $w$, then so is the intersection.) 
In particular $Q_\alpha$
is ccc.

We call this forcing ``partial random forcing defined from $w$''.
Analogously, we define the ``partial Hechler'',
``partial eventually different''\footnote{See \ref{def.eor} for the definition.}
and ``partial amoeba'' forcings (and the same argument 
shows that these forcings are also ccc).

Assume that $\lambda $  is regular uncountable and that $\mu<\lambda$
implies $\mu^{\aleph_0}<\lambda$.
Then $|w|<\lambda$ implies that the sizes of the partial forcings
defined by $w$ are
${<}\lambda$.

We will assume the following throughout the paper:
\begin{assumption}\label{asm:P}
$\aleph_1<\lambda_1<\lambda_2<\lambda_3<\lambda_4<\lambda_5$ are regular cardinals
such that $\mu<\lambda_i$
implies $\mu^{\aleph_0}<\lambda_i$. 
Furthermore, 
$\lambda_3$ is the successor of a regular cardinal $\chi$ with $\chi^{\aleph_0}=\chi$,
and $\lambda_5^{<\lambda_4}=\lambda_5$.

We set $\delta_5=\lambda_5+\lambda_5$, and 
partition $\delta_5\setminus\lambda_5$ into unbounded sets $S^1$, $S^2$, $S^3$ and $S^4$.
Fix for each $\alpha\in \delta_5\setminus\lambda_5$ some $w_\alpha\subseteq \alpha$ such that each
$\{w_\alpha:\, \alpha\in S^i\}$ is cofinal\footnote{i.e.,
if $\alpha\in S^i$ then $|w_\alpha|<\lambda_i$,
and for all $u\subseteq \delta_5$, $|u|<\lambda_i$
there is some $\alpha\in S^i$ with $w_\alpha\supseteq u$.} 
in $[\delta_5]^{{<}\lambda_i}$.
\end{assumption}
The reader can assume that $(\lambda_i)_{i=1,\dots,5}$,
$(S^i)_{i=1,\dots,4}$  as well as
$(w_\alpha)_{\alpha\in S^i}$ for $i=1,2,3$ have been fixed once and 
for all (let us call them ``fixed parameters''), 
whereas we will investigate various possibilities for
$\bar w=(w_\alpha)_{\alpha\in S^4}$ in the following Subsections~\ref{ss:no-gch} and \ref{ss:recovering}.
(We will call such a $\bar w$ that satisfies the assumption a ``cofinal parameter''.)

\begin{definition}\label{def:Pa}
Let
$\Pa=(P_\alpha,Q_\alpha)_{\alpha\in\delta_5}$ 
be the 
FS iteration where 
$Q_\alpha$ is Cohen forcing for $\alpha\in \lambda_5$, and 
\begin{flalign*}
&&
Q_\alpha\text{ is the partial }
\left\{
\begin{array}{c}
\text{amoeba}\\
\text{random}\\
\text{Hechler}\\
\text{eventually different}\\
\end{array}\right\}
\text{ forcing defined from $w_\alpha$ if $\alpha$ is in}
\left\{
\begin{array}{l}
S^1\\
S^2\\
S^3\\
S^4\\
\end{array}
\right.
\\
\end{flalign*}
\end{definition}

According to Lemma~\ref{lem:gettinggood} $\Pa$ is  $(\lambda_i,\Rel _i)$-good
for $i=1,2,4$, 
so Lemmas~\ref{lem:coboundedunbounded} and~\ref{lem:linearcharacteristics} give us:

\begin{lemma}\label{lem:linearPapartial}
$\mylini(\Pa,\kappa)$ holds for $i=1,2,4$ and each regular cardinal $\kappa$ in $[\lambda_i,\lambda_5]$.
\end{lemma}

So in particular, $\Pa$ forces $\addnull\le\lambda_1$, $\covnull\le\lambda_2$, $\nonmeager\le\lambda_4$ and $\covmeager=\nonnull=\cofnull=\lambda_5=2^{\aleph_0}$;
i.e., the respective characteristics in the left half of Cicho\'n's diagram are small enough. 
It is easy to see that they are also large enough:

For example,  the partial amoebas and the fact that
$(w_\alpha)_{\alpha\in S^1}$ is cofinal ensure that
$\Pa$ forces $\addnull\ge \lambda_1$:
Let $(N_k)_{k\in\mu}$, $\aleph_1\le\mu<\lambda_1$ be a family of $\Pa$-names of null sets.
Each $N_k$ is a Borel-code, i.e., a real, i.e., 
a sequence of natural numbers, each of which is decided by a maximal antichain
(labeled with natural numbers). Each condition in such an antichain has finite support,
hence only uses finitely many coordinates in $\delta_5$.
So all in all we get a set $w^*$ of size ${\le}\mu$ that already decides all $N_k$.
(I.e., for each $k\in \mu$ there are a Borel function $B$ in $V$ and 
a sequence $(\alpha_j)_{j\in\omega}$ in $V$ of elements of $w^*$ 
such that $N_k=B(\eta_{\alpha_0},\eta_{\alpha_1},\dots)$.)
There is some $\beta\in S^1$ such that $w_\beta\supseteq w^*$, 
and the partial amoeba forcing at $\beta$ sees all the null sets $N_k$ and therefore 
covers their union.

We will reformulate this in a slightly cumbersome manner
that can be conveniently used later on, using the ``cone of bounds'' property $\mypart$:

\begin{definition}\label{def:partial}
For a ccc forcing notion $P$, regular uncountable cardinals $\lambda,\mu$
and 
$i=1,3,4$, let $\myparti(P,\lambda,\mu)$ stand for:
\begin{quote}
There are a ${<}\lambda$-directed partial order $(S,\prec)$
of size $\mu$
and a sequence $(g_s)_{s\in S}$ of $P$-names for reals
such that for each $P$-name $f$ of a real
$(\exists s\in S)\,(\forall t\succ s)\, P\forces f \Ri g_t $.
\end{quote}
\end{definition}
So $s$ is the tip of a cone that consists of elements bounding $f$.
\begin{lemma}\label{lem:temp}
For $i=1,3,4$,
$\myparti(P,\lambda,\mu)$ implies 
$P\forces(\, \mathfrak{b}_i\ge\lambda \,\&\, \mathfrak{d}_i\le \mu\,)$.
\end{lemma}

\begin{proof}
$\mathfrak{d}_i\le \mu$, as the set $(g_s)_{s\in S}$ is a dominating family of size $\mu$.
To show $\mathfrak{b}_i\ge\lambda$, 
assume $(f_\alpha)_{\alpha\in\theta}$  is a sequence of $P$-names
of length $\theta<\lambda$.
For each $f_\alpha$ there is a cone of upper bounds with tip $s_\alpha\in S$,
i.e., $(\forall t\succ s_\alpha)\, P\forces f_\alpha \Ri g_t $.
As $S$ is ${<}\lambda$-directed, there is some 
$t$ above all tips $s_\alpha$.
Accordingly, $P\forces f_\alpha \Ri g_t $ for all $\alpha$,
i.e., $\{f_\alpha:\, \alpha\in\theta\}$ is not unbounded.
\end{proof}

So for example, $\mypartI(P,\lambda,\mu)$ implies $\lambda_1\le\mathfrak{b}_1=\addnull$, etc.
The definition and lemma would work for $i=2$ as well, but would not be useful\footnote{More specifically: this definition would give us the property $g_t\notin f$ only for the null sets of the 
specific form $f=\{h:\, \lnot r \RII h\}=N_r$ for some $r\in 2^\omega$;
whereas we will define $\mypartII$ to deal with all names $f$ of
null sets.}
as we do not have $\mathfrak{b}_2\le\covnull$.
So instead, we define $\mypartII$ separately:

\begin{definition}\label{def:partial2}
For $P$, $\lambda$ and $\mu$ as above,  let $\mypartII(P,\lambda,\mu)$ stand for:
\begin{quote}
There are a ${<}\lambda$-directed partial order $(S,\prec)$
of size $\mu$
and a sequence $(g_s)_{s\in S}$ of $P$-names for reals
such that for each $P$-name $f$ of a null set 
$(\exists s\in S)\,(\forall t\succ s)\, P\forces g_t\notin f$.
\end{quote}
\end{definition}

\begin{lemma}\label{lem:partialcharacteristics}
\begin{itemize}
\item[1.] $\mypartI(P,\lambda,\mu)$ implies $P\forces(\, \addnull\ge\lambda \,\&\, \cofnull\le \mu\,)$.
\item[2.] $\mypartII(P,\lambda,\mu)$ implies $P\forces(\, \covnull\ge\lambda \,\&\, \nonnull\le \mu\,)$.
\item[3.] $\mypartIII(P,\lambda,\mu)$ implies $P\forces(\, \mathfrak{b}\ge\lambda \,\&\, \mathfrak{d}\le \mu\,)$.
\item[4.] $\mypartIV(P,\lambda,\mu)$ implies $P\forces(\, \nonmeager\ge\lambda \,\&\, \covmeager\le \mu\,)$.
\end{itemize}
\end{lemma}
\begin{proof}
The cases $i\neq 2$ are  direct consequences of Lemmas~\ref{lem:connection} and~\ref{lem:temp}. The proof
for $i=2$ is analogous to the proof of Lemma~\ref{lem:temp}.
\end{proof}

\begin{lemma}\label{lem:partialPa}
$\myparti(\Pa,\lambda_i,\lambda_5)$ holds
(for $i=1,2,3,4$).
\end{lemma}

\begin{proof}
Set $S=S^i$ and 
$s\prec t$ if $w_s\subsetneq w_t$. As $\lambda_i$ is regular, $(S,\prec)$ is ${<}\lambda_i$-directed. 
Let $g_s$ be the generic added at $s$ (e.g., the partial random real in case of $i=2$, etc).
A $\Pa$-name $f$ depends (in a Borel way) on 
the subsequence of generics indexed by a countable set $w^*\subseteq \delta$.
Fix some $s\in S^i$ such that $w_s\supseteq w^*$.
Pick any $t\succ s$. Then $w_t\supseteq w_s$, so $w_t$ contains all information to calculate
$f$, so we can show that $P\forces f \Ri g_t $. Let us list the possible cases:
$i=2$: A partial random real $g_t$ will avoid the
null set $f$.
$i=3$: A partial Hechler real $g_t$ will dominate $f$.
$i=4$: A partial eventually different real $g_t$ will be eventually different from $f$.
As for $i=1$, we use\footnote{Alternatively, we could use, instead of amoeba, some other Suslin ccc forcing that more directly adds an $\RI$-dominating element of $\mathcal C$.} 
Lemma~\ref{lem:connection2}.
\end{proof}

So to summarize what we know so far about $\Pa$:
\begin{itemize}
\item $\myparti$ holds for $i=1,2,3,4$. So the left hand characteristics are large.
\item $\mylini$ holds for $i=1,2,4$. So the left hand characteristics other than $\mathfrak{b}$ are small.
\end{itemize}
However, $\mylinIII$ (corresponding to ``$\mathfrak{b}$ small'') is missing; and we cannot get it by a simple ``preservation
of $(\RIII,\lambda_3)$-goodness'' argument.
Instead, we will argue in the following two sections that it is possible to choose 
the parameter $(w_\alpha)_{\alpha\in S^4}$ in such a way that $\mylinIII$ holds as well.

\subsection{Dealing with \texorpdfstring{$\mathfrak{b}$}{b} without GCH}
\label{ss:no-gch}
In this section, we follow (and slightly modify)
the main construction of~\cite{MR3513558}.

In this section (and this section only) we will assume  the following
(in addition to Assumption~\ref{asm:P}, i.e., in particular to the assumption $\lambda_3=\chi^+$):
\begin{assumption}\label{asm:chi}(This section only.) $2^\chi =|\delta_5|= \lambda_5$.
\end{assumption}

Set $S^0=\lambda_5\cup S^1\cup S^2\cup S^3$.
So $\delta_5=S^0\cup S^4$,
and $\Pa$ is a FS ccc iteration along $\delta_5$ such that $\alpha\in S^0$
implies $|Q_\alpha|<\lambda_3$, i.e., $|Q_\alpha|\le\chi$.
Let us fix $P_\alpha$-names 
\begin{equation}\label{eq:ia}
i_\alpha:Q_\alpha\to\chi\text{ injective}
\end{equation}
(for $\alpha\in S^0$). Note that we can 
strengthen each $p\in \Pa$ to some $q$ such that
$\alpha\in \supp(q)\cap S^0$ 
implies $q\restriction\alpha\forces i_\alpha(q(\alpha))=\check \jmath$   for some 
 $j\in\chi$.

For $\alpha\in S^4$, $Q_\alpha$ is a partial eventually different
forcing. At this point, we should specify which 
variant of this forcing we actually use:\footnote{In the previous section it did not matter which variant we use.}
\begin{definition}\label{def.eor}
\begin{itemize}
\item
Eventually different forcing $\Eor$ consists of all
tuples $(s,k,\varphi)$, where $s\in \omega^{<\omega}$, $k\in \omega$, 
and $\varphi:\omega\to [\omega]^{\le k}$ satisfies $s(i)\notin \varphi(i)$
for all $i\in \dom(s)$. 
\item
We define $(s',k',\varphi')\le (s,k,\varphi)$ if $s\subseteq s'$, 
$k\le k'$, 
and $\varphi(i) \subseteq \varphi'(i)$ for all $i$.  
\item
The generic object $g^*=\bigcup_{(s,k,\varphi)\in G_\Eor}{s}$ is a function such that each condition $(s,k,\varphi)$
forces that $s$ is an initial segment of $g^*$, and $g^*(i)\notin \varphi(i)$ 
for all $i$.   
\item
We call $s\in \omega^{{<}\omega}$ the ``stem'' of $(s,k,\varphi)$ and $k\in\omega$ the ``width''.
\end{itemize}
\end{definition}

A density argument shows that $g^*$ will be eventually different from 
all functions $f:\omega\to \omega$ from $V$. 

The following is easy to see:
\begin{itemize}
    \item \label{compat.a} If $p,q\in \Eor$ are compatible, then they have a greatest lower bound. 
    \item \label{compat.b} Any finite set of conditions with the same stem has a lower bound (again with the same stem). 
     So $ \Eor$ is $\sigma$-centered.
     \item If $q=(s',k',\varphi')$ and 
     $p=(s,k,\varphi)$ and $s'$ extends $s$,
     then $p$ and $q$ are compatible iff
     $s'(i)\notin \varphi(i)$ for all $i\in\dom(s')$.
     \item  \label{compat.c}  If a condition $q^* =(s^*,k^*,\varphi^ *)$ is compatible with each condition in a finite set $B\subseteq \Eor$, and $s^*$ extends $s$ for each $(s,k,\varphi)\in B$, then 
     the set $B\cup \{q^*\}$ has a lower bound. 
     (Use  $s^*$ as stem, and take the pointwise union of all $\varphi$ that occur in $B\cup \{q^ *\}$.)
\end{itemize}
We will not force with $\Eor$, but 
with a partial version of $\Eor$.
In the $P_\alpha$-extension (for $\alpha\in S^4$), 
this partial forcing $Q_\alpha=\Eor'$ is a 
(generally not complete) sub-forcing of $\Eor$ which is easily
seen to be closed under conjunctions
(i.e., under the partial operation ``greatest lower bound'' of finite sets of conditions). 
Note that this implies that compatibility is absolute between $\Eor$ and $\Eor'$,
and that the previous items also hold for $\Eor'$. For later reference, let
us explicitly state the last item:
\begin{fact}\label{fact:bla}
Assume $\Eor'\subseteq \Eor$ is closed under 
conjunctions. 
If a condition $q^* =(s^*,k^*,\varphi^ *)$ in $\Eor'$ is compatible with each condition in a finite set $B\subseteq \Eor'$, and $s^*$ extends $s$ for each $(s,k,\varphi)\in B$, then 
     the set $B\cup \{q^*\}$ has a lower bound in $\Eor'$. 
\end{fact}

\begin{definition}
   Let $D$ be a non-principal ultrafilter on $\omega$, and let $\bar p = 
(p_n)_{n\in \omega} = (s,k,\varphi_n)_{n\in \omega}$ be a
sequence of conditions in $\Eor$ with the same stem and the same  width. 
We define $\lim_D\bar p$ to be  $(s,k,\varphi_\infty)$, 
where for all $i$ and all $j$ we have 
$
 j\in \varphi_\infty(i) \Leftrightarrow \{ n: j\in \varphi_n(i)\}\in D
$. 
\end{definition}
The following is easy to see: 
$\lim_D\bar p\in\Eor$ and
if $q\le \lim_D \bar p$, then the set $B\coloneq\{ n\in \omega:\, p_n \text{ compatible with } q\}$ is in $D$.

(Proof: $q=(s',k',\varphi')\le \lim_D\bar p=(s,k,\varphi_\infty)$. So for each $i\in\dom(s')$, $s'(i)\notin \varphi_\infty(i)$, and by the definition of the limit, $A^i \coloneq\{n:\, s'(i)\notin \varphi_n(i)\}\in D$. If $n\in\bigcap_{i\in\dom(s')} A^i$, then $p_n$ is compatible with $q$.)

As $B$ is defined using only compatibility,
the statement still holds for
compatibility preserving subforcings.
We state it for later reference in the 
following form:
\begin{fact}\label{fact:blubb4}
Assume that\/  $\Eor'$  is a subforcing
of\/ $\Eor$ closed under conjunctions,
let $\bar p$ be a sequence of\/ $\Eor'$
conditions with the same stem and width, and 
assume that $\lim_D(\bar p)\in \Eor'$
and that $q\le_{\Eor'} \lim_D(\bar p)$.
Then 
$B\coloneq\{ n\in \omega:\, p_n \text{ compatible with } q\}$ is in $D$.
\end{fact}




\begin{definition}
\begin{itemize}
   \item A ``partial guardrail'' is a 
function $h$ defined on a subset of $\delta_5$  
such that $h(\alpha)\in \chi$ for $\alpha\in S^0\cap\dom(h)$, and $h(\alpha)
\in \omega^{<\omega}\times \omega$ for $\alpha\in S^4\cap\dom(h)$.
  \item A ``countable guardrail'' is a partial guardrail with
  countable domain. A ``full guardrail'' is a partial guardrail with domain $\delta_5$.
\end{itemize}
\end{definition}  

We will use the following lemma, which is a consequence of the Engelking-Kar{\l}owicz theorem~\cite{MR0196693} on the density of box products (cf.~\cite[5.1]{MR3513558}):  
\begin{lemma}\label{use.EK}
(As $|\delta_5|\le 2^\chi$ and $\chi^{\aleph_0}=\chi$.) There is a family $H^*$ of full guardrails
with $|H^*|=\chi$,
such that each countable guardrail is extended by
some $h\in H^*$. 
We will fix such an $H^*$ and  enumerate it as $(h^*_\ve)_{\ve\in \chi}$.
\end{lemma}

Note that the notion of guardrail (and the density property required in
Lemma~\ref{use.EK}) only
depends on $\chi$, $\delta_5$, $S^0$ and $S^4$, i.e., on fixed parameters;
so we can fix an $H^*$ that will work for all cofinal parameters $\bar w=(w_\alpha)_{\alpha\in S^4}$.

Once we have decided on $\bar w$, and thus have defined $\Pa$, we can define the following:
\begin{definition}
 A condition $p\in \Pa$ follows the full guardrail $h$, 
  if 
  \begin{itemize}
   \item 
   for all $\alpha\in S^0 \cap \dom(p)$, the empty condition of $P_\alpha$ forces 
   that $p(\alpha)\in Q_\alpha$ and $i_\alpha(p(\alpha)) = h(\alpha)$ (where $i_\alpha$ is defined in~\eqref{eq:ia}), and
  \item  for all $\alpha\in S^4\cap \dom(p)$:
  \begin{itemize}
  	\item  $p\on\alpha $ forces
      that the pair of stem and width of $p(\alpha) $ is equal to $h(\alpha)$, and moreover
    \item $p$ is determined at $\alpha$. (This was defined in~\ref{def:groundmodelsequence}(\ref{item:det}): We already know in $V$ a code $(B,u)$ that evaluates to $p(\alpha)$.)
     \end{itemize}
  \end{itemize}       
\end{definition}
As we are dealing with a FS iteration,
the set of conditions $p$ determined at each 
position $\alpha\in\dom(p)$ is easily seen to be dense (by induction).
So note that     
      \begin{itemize}
      \item      the set of conditions $p$ such that there is \emph{some} guardrail $h$
      such that $p$ follows $h$, is dense; while
      \item 
      
      for each \emph{fixed} guardrail $h$, the set      
      of all conditions $p$ following $h$, is \emph{centered} (i.e.,
      each finitely many such $p$ are compatible).
      \end{itemize}
 
\begin{definition}
\begin{itemize} 
  \item A ``$\Delta$-system with root $\nabla$ following the full guardrail $h$'' is a family $\bar p=(p_i)_{i\in I}$ 
     of conditions all following $h$, where 
       $(\dom(p_i):i\in I)$ is a $\Delta$-system with root $\nabla$ in the usual sense (so $\nabla\subseteq \delta_5$ is finite).
   \item 
We will be particularly interested in countable $\Delta$-systems.  
   Let $(p_n:n\in \omega)$ be such a $\Delta$-system with root
   $\nabla$ following $h$,  and assume that $\bar
   D=(D_\alpha:\alpha\in u)$ is  
  a sequence such that $u\supseteq \nabla \cap S^4$
  and each $D_\alpha$ 
   is a $P_\alpha$-name
   of an ultrafilter on $\omega$. Then we define the $\lim_{\bar D}\bar p$
   to be the following function with domain $\nabla$: 
   \begin{itemize}
   \item If $\beta\in \nabla\cap S^0 $, then $\lim_{\bar D}\bar p(\beta)$ is the common 
   value of all $p_n(\beta)$.  (Recall that this value is already determined by 
    the guardrail $h$.)
   \item If $\alpha\in \nabla\cap S^4$, then $\lim_{\bar D}\bar p(\alpha)$ is
    (forced by $\Pa_\alpha$ to be) $\lim_{D_\alpha}(p_n(\alpha))_{n\in \omega}$.  
    
   \end{itemize}
\end{itemize}
\end{definition}

Note that in general $\lim_{\bar D}\bar p$ will not be a condition in
$\Pa$: 
For $\alpha\in S^4\cap \nabla$,
the object $\lim_{\bar D}\bar p(\alpha)$ will be forced to be in the eventually different forcing $\Eor$, but not 
necessarily in the \emph{partial} eventually different forcing $Q_{\alpha}\subseteq\Eor$. 

Also note the following:
If $\bar p$ is a countable $\Delta$-system, and 
$\alpha\in\nabla\cap S^4$, then
$(p_n(\alpha))_{n\in\omega}$ is a ground-model-code-sequence
(see Definition~\ref{def:groundmodelsequence}(\ref{item:gms})). This follows trivially from the 
definition of ``$p_n$ follows $h$'' and the fact that 
$\bar p$ is in $V$. 

Recall that we assume all of the parameters defining $\Pa=(P_\alpha,Q_\alpha)_{\alpha\in \delta_5}$ to be fixed,
apart from $(w_\alpha)_{\alpha\in S^4}$.
Once we fix $w_\alpha$ for $\alpha\in S^4\cap \beta$, we know $P_\beta$.
\begin{construction}\label{constr}
We can construct by induction on $\alpha\in \delta_5$ 
the sequences $(D^\ve_\alpha)_{\ve\in \chi}$ and, if $\alpha\in S^4$,
also $w_\alpha$, such that:
\begin{enumerate}[(a)]
\item\label{item:olda} Each $D^\ve_\alpha$ is a $P_\alpha$-name of a nonprincipal ultrafilter extending 
 $\bigcup_{\beta<\alpha}D^\ve_\beta$.
\item\label{item:oldb1} For each countable $\Delta$-system $\bar p$ in $P_\alpha$ which follows
the guardrail $h^*_\ve\in H^*$:
\\
$\lim_{(D^\ve_\beta)_{\beta<\alpha}}\bar p$ is in $P_\alpha$\,\dots
\item\label{item:oldb2} \dots\ and forces that 
$A_{\bar p}\coloneq \{n\in\omega:\, p_n\in G_\alpha\}$ is in $D^\ve_\alpha$.
\item\label{item:oldc} (If $\alpha\in S^4$) 
$w_\alpha\subseteq \alpha$, $|w_\alpha|<\lambda_4$, and
%
%
%
for all ground-model-code-sequences\footnote{see Definition~\ref{def:groundmodelsequence}(\ref{item:gms}).}
for elements of $Q_\alpha$, the 
$D^\ve_\alpha$-limit is forced to be in $Q_\alpha$
as well (for all $\ve\in \chi$).

(Actually, the set of $w_\alpha$ satisfying this
is an $\omega_1$-club set in $[\alpha]^{{<}\lambda_4}$.\footnote{I.e., for each $w^*\in [\alpha]^{{<}\lambda_4}$ there is a $w_\alpha\supseteq w^*$ satisfying (d), and if $(w^i)_{i\in\omega_1}$ is an increasing sequence of sets satisfying (d), then the limit $w_\alpha\coloneq \bigcup_{i\in\omega_1}w^i$ satisfies (d) as well.})
\end{enumerate}
\end{construction}
%
%
\begin{proof}

\emph{(\ref{item:oldb1}) for $\alpha$ limit:}
The root of a $\Delta$-system is finite and therefore below some $\beta<\alpha$, so the limit exists (by induction) already in $P_\beta$. 

\emph{(\ref{item:olda}+\ref{item:oldb2}) for $\alpha$ limit:}
It is enough to show, for each $\ve\in\chi$, that $P_\alpha$
forces that the following generates a proper filter (i.e., any finite intersection
of elements of this set is nonempty):
\[
\bigcup_{\beta <\alpha} D^\ve_\beta\ \cup\ \{ A_{\bar p}:\,
\bar p\text{ is a countable $\Delta$-system following $h^*_\ve$ and }\lim\nolimits_{(D^\ve_\beta)_{\beta<\alpha}}\bar p\in G_\alpha\}.
\]
(Then we let $D^\ve_\alpha$ be any ultrafilter extending this set.)

So assume towards a contradiction that $q\in P_\alpha$ forces that 
$A\cap A_{\bar p^0}\cap \dots\cap A_{\bar p^{n-1}}=\emptyset$, where
$A\in D^\ve_{\beta_0}$
for some $\beta_0<\alpha$ (we can assume $\beta_0$ is already decided in $V$) and $\bar p^i$ as above
with $q\le \lim\nolimits_{(D^\ve_\beta)_{\beta<\alpha}}\bar p^i $ for $i<n$.
Let $\beta_1<\alpha$ be the maximum of the union of the roots
of the $\bar p^i$, and set 
$\beta_2\coloneq max(\supp(q))$ and
$\gamma\coloneq max(\beta_0,\beta_1,\beta_2)+1$.
By the induction hypothesis, $q$ forces  
$A'\coloneq A\cap \bigcap_{i<n} A_{\bar p^i\restriction \gamma}\in D^\ve_\gamma$ (as $\lim\nolimits_{(D^\ve_\beta)_{\beta<\gamma}}\bar p^i\restriction\gamma  = \lim\nolimits_{(D^\ve_\beta)_{\beta<\alpha}}\bar p^i$, since the root lies below $\gamma$).
As $A'$ is a  $P_\gamma$-name, we can find 
$q'\le q$ in $P_\gamma$ and $\ell\in\omega$ such that 
$q'\forces \ell\in A'$.
We now find $q''\le q'$ in $P_\alpha$ by defining $q''(\beta) $ for
each element $\beta$  of the finite set 
$\bigcup_{i<n}\supp (p^i_\ell)\setminus \gamma$: For such $\beta$ in $S^0$, the guardrail gives a specific 
value $h^*_\ve(\beta)\in Q_\beta$, which we use for $q''(\beta)$ as well.
For $\beta\in S^4$, 
all conditions $p^i_{\ell}(\beta)$ (where defined) 
have the same stem and width $h^*_\ve(\beta)$; hence there is a common
extension $q''(\beta)$.

Clearly $q''$ forces that $\ell$ is in the allegedly empty set, the desired contradiction.

\emph{(\ref{item:oldb1}) for $\alpha=\gamma+1$ successor:}
Assume the nontrivial case, $\gamma\in S^4$:
Write the $\Delta$-system as $(p_i,q_i)_{i\in\omega}$
with $(p_i,q_i)\in P_\gamma*Q_\gamma$.
As noted above, $(q_n)_{n\in\omega}$ is a ground-model-code-sequence, and by induction 
(\ref{item:oldc}) holds for $w_\gamma$.
So it is forced that
the $D^\ve_\gamma$-limit $q^*$ of the $q_n$ is in $Q_\gamma$.
Again by induction, the limit $p^*$ of the $p_n$ exists as well;
and $(p^*,q^*)$ is the required limit.

\emph{(\ref{item:olda}+\ref{item:oldb2}) for $\alpha=\gamma+1$ successor:}
We again have to show that $P_\alpha$ forces that the following is a filter base, for each $\ve\in\chi$:
\[
D^\ve_\gamma\cup\{A_{\bar p}:\, \bar p\text{ is a countable $\Delta$-system following $h^*_\ve$ and }\lim\nolimits_{(D^\ve_\beta)_{\beta<\alpha}} \bar p\in G_\alpha\}.
\]
As above, assume that $q$ forces $A\cap A_{\bar p^0}\cap\dots\cap A_{\bar p^{n-1}}=\emptyset$.

We can assume that $q\restriction \gamma$ forces that $q(\gamma)$ is stronger 
than the limit of all $\bar p^i(\gamma)$ (for $i<n$).
Thus,  by Fact~\ref{fact:blubb4}, each  $B_i \coloneq \{\ell\in\omega:\, q(\gamma) \text{ compatible with  } p^ i_\ell(\gamma)\}$ is forced to be in $D_\gamma^ \varepsilon$. 


By induction, $q\restriction \gamma$ forces  that
$A'\coloneq A\cap \bigcap_{i<n} A_{\bar p^i\restriction \gamma}\in D^\ve_\gamma$,
and therefore also forces that $B'=A'\cap \bigcap_{i<n} B_i$ is in the ultrafilter and in particular nonempty. 
Work in the 
$P_\gamma$-extension by some generic filter
containing
$q\restriction \gamma$. Fix some $\ell\in B'$. 
%
%
By the definition of $B_i$, $q(\gamma)$ is compatible with each $p^i_\ell(\gamma)$ for $i<n$.
According to Fact~\ref{fact:bla}
there is a common lower bound $q''$.


$q\restriction \gamma \Vdash_{P_\gamma} 
q'' \Vdash_{Q_\gamma}\ell\in A_{\bar p^i}$.
I.e.,
$q\restriction \gamma*q''\le q$ forces that 
$\ell$ is an element of the allegedly empty set. 


\emph{(\ref{item:oldc})} For any $w\subseteq\alpha$, let $Q^w$ be the 
($P_\alpha$-name for) the partial eventually different forcing
defined using $w$. 
Start with some $w^0\subseteq \alpha$ of size ${<}\lambda_4$. 
There are $|w^0|^{\aleph_0}$ many ground-model sequences 
in $Q^{w^0}$.
For any $\ve$ and any such sequence, the $D^\ve_\alpha$-limit 
is a real; so we can extend $w^0$ by a countable set to some $w'$ such that
$Q^{w'}$ contains the limit. We can do that for all $\ve\in\chi$
and all sequences, resulting in some $w^1\supseteq w^0$
still of size ${<}\lambda_4$. We iterate this construction and 
get $w^i$ for $i\le\omega_1$, taking the unions at limits.
Then $w_\alpha\coloneq w^{\omega_1}$ is as required, as $Q_\alpha\coloneq Q^{w_\alpha}=\bigcup_{i<\omega_1}Q^{w_i}$.

So this proof actually shows that the set of  $w_\alpha$ with the desired property is an $\omega_1$-club.
\end{proof}

After carrying out the construction of this lemma, we get a forcing
notion $\Pa$ satisfying the following: 

\begin{lemma}
$\mylinIII(\Pa,\kappa)$ for $\kappa\in[\lambda_3,\lambda_5]$,
witnessed by the sequence $(c_\alpha)_{\alpha<\kappa}$ of the first $\kappa$ many Cohen reals.
\end{lemma}

\begin{proof}
We want to show that for every $\Pa$-name $y$
there are coboundedly many $\alpha\in\kappa$ such that 
$\Pa\forces \lnot c_\alpha \le^* y$.

Assume that $p^*$ forces that 
there are unboundedly many $\alpha\in\kappa$
with $c_\alpha\le^* y$, and enumerate them as $(\alpha_i)_{i\in\kappa}$ 
in increasing order (so in particular $\alpha_i\ge i$). 
Pick $p_i\le p^*$ deciding $\alpha_i$ to be some $\beta_i$, and also deciding 
$n_i$ such that $(\forall m\ge n_i)\, c_{\alpha_i}(m)\le y(m)$.
We can assume that $\beta_i\in\dom(p_i)$.
Note that $\beta_i$ is a Cohen position (as $\beta_i<\kappa\le \lambda_5$), and we can assume that $p_i(\beta_i)$ is a Cohen condition in $V$ (and not just a $P_{\beta_i}$-name for such a 
condition).
By thinning out, we may assume:
\begin{itemize}
\item
All $n_i$ are equal to 
some $n^*$.
\item $(p_i)_{i\in\kappa}$ forms a 
$\Delta$-system with root $\nabla$.
\item $\beta_i \notin \nabla$, hence all $\beta_i$ are distinct.

(For any $\beta\in\kappa$, at most 
$|\beta|$ many $p_i$ can force $\alpha_i=\beta$, as 
$p_i$ forces that $\alpha_i\ge i$ for all $i$.)
\item $p_i(\beta_i)$ is always the same Cohen 
condition $s$, 
without loss of generality of length $n^{**}\ge n^*$.

(Otherwise extend $s$.)
\end{itemize}
Pick the first $\omega$ many elements $(p_i)_{i\in\omega}$ 
of this $\Delta$-system.
Now extend each $p_i$ to $p_i'$ by extending
the Cohen condition $p_i(\beta_i)=s$ to
$s^\frown i$ (i.e., forcing $c_{\alpha_i}(n^{**})=i$).
Note that 
$(p'_i)_{i\in\omega}$ is still a countable $\Delta$-system, following some new countable guardrail
and therefore some full guardrail $h^*_\ve\in H^*$.

Accordingly, the limit 
$\lim_{(D^\ve_\alpha)_{\alpha\in \delta_5}}\bar p'$
forces that infinitely many of the
$p'_i$ are in the generic filter.
But each such $p'_i$ forces that $c_{\alpha_i}(n^{**})=i \le y(n^{**})$,  a contradiction.
\end{proof}

\subsection{Recovering GCH}
\label{ss:recovering}
For the rest of the paper we will assume the following for
the ground model $V$
(in addition to Assumption~\ref{asm:P}):
\begin{assumption}
GCH holds.
\end{assumption}
(Note that this is incompatible with Assumption~\ref{asm:chi}.)

Recall that all parameters used to define $\Pa$ are fixed, 
apart from $\bar w=(w_\alpha)_{\alpha\in S^4}$.

\begin{lemma}\label{lem:gch-construction}
We can choose $\bar w$ such that
$\mylinIII(\Pa,\kappa)$ holds for all regular $\kappa\in [\lambda_3,\lambda_5]$.
\end{lemma}

For the proof, we will use the following easy observation:
\begin{lemma}\label{lem:club}
Assume $\chi$ is a cardinal and $B$
a set and $X^0\in [B]^\chi$, 
$\mR$ is a $\chi^+$-cc forcing notion,
and $C$ is an $\mR$-name such that
the empty condition forces 
that $C$ is an $\omega_1$-club
subset of $[B]^\chi$.
Then there is a set $X\supseteq X^0$ (in the ground model) such that the empty condition forces 
$X\in C$.
\end{lemma}
\begin{proof}
By induction, choose (in the ground model) sequences 
$X^\alpha, \tilde X^\alpha$ for $\alpha< \omega_1$ such that
$X^\alpha$ is in $[B]^\chi$, the sequence of the $X^\alpha$ is increasing with $\alpha$, $\tilde X^\alpha$
is an $R$-name, and the empty condition forces:
``$\tilde X^\alpha$ is in $C$ and is a  superset of $X^\alpha$;
and the sequence of the $\tilde X^\alpha$ is increasing (not necessarily continuous).''
Moreover, the empty condition forces $\tilde X^\alpha\subseteq X^{\alpha+1}$.
(In a  limit step $\gamma$, we set $X^\gamma=\bigcup_{\alpha<\gamma}X^\alpha$, and in a successor step $\alpha+1$ we use $\chi^+$-cc to cover the name $\tilde X^\alpha$.)
Then $X=\bigcup_{\alpha\in\omega_1}X^\alpha$ is as required.
\end{proof}

\begin{proof}[Proof of Lemma~\ref{lem:gch-construction}]
Let $\mR$ be a ${<}\chi$-closed $\chi^+$-cc p.o.\ that forces
$2^\chi=\lambda_5$.

In the $\mR$-extension $V^*$, Assumption~\ref{asm:chi} holds;
and Assumption~\ref{asm:P} still holds for the fixed parameters.\footnote{In particular, $(w_\alpha)_{\alpha\in S^i}$ is still cofinal in
$[\delta_5]^{{<}\lambda_i}$: For $i=1,2$, the forcing $\mR$
doesn't add any new elements of $[\delta_5]^{{<}\lambda_i}$ as 
$\mR$ is $\lambda_i$-closed;
for $i=3$ any new subset of $\delta_5$ of size $\theta<\lambda_3$
is contained in a ground model set of size at most $\theta\times \chi<\lambda_3$, as $\mR$ is $\chi^+$-cc.}

So in $V^*$, we can perform the inductive Construction~\ref{constr}, where now ``ground model'' refers to $V^*$, not $V$ (e.g.,
when we talk about determined positions,
or ground-model-code-sequences, etc). 
Actually, we can construct in $V$ the following, by induction on $\alpha\in \delta_5$, and starting
with some cofinal $\bar w^{\text{initial}}=(w^{\text{initial}}_\alpha)_{\alpha\in S^4}$ in $V$:
\begin{itemize}
\item An $\mR$-name $(D^\ve_\alpha)_{\ve\in \chi}$ (forced to be constructed) according to \ref{constr}(a,b,c).
\item If $\alpha\in S^4$, 
some $w_\alpha\supseteq w^\text{initial}_\alpha$ in $V$
such that $\mR$ forces $w_\alpha$ satisfies \ref{constr}(d).

(We can do this by Lemma~\ref{lem:club}, as the set of potential $w_\alpha$'s is an $\omega_1$-clubset of $[\alpha]^{{<}\lambda_4}$.)
\end{itemize}

So we get in $V$ a cofinal parameter $\bar w$ satisfying the following: 
%
In the $\mR$-extension $V^*$, the same parameters
define a forcing (call it 
$\PaB$)
satisfying 
$\mylinIII(\PaB,\kappa)$ in $V^*$.

$\PaB$ is basically the same as $\Pa$. More formally:
\begin{quote}
In the $\mR$-extension $V^*$,  $\Pa=(P_\alpha,Q_\alpha)_{\alpha<\delta_5}$ (the iteration constructed in $V$) is 
canonically densely embedded into $\PaB=(P^*_\alpha,Q^*_\alpha)_{\alpha<\delta_5}$ 
(the iteration constructed in $V^*$ using the same parameters).
\end{quote}

Proof: By induction, we show (in the $\mR$-extension) that $P^*_\alpha$ forces 
that $Q^*_\alpha$ (evaluated by the $P^*_\alpha$-generic) is equal to
$Q_\alpha$ (evaluated by the induced $P_\alpha$-generic, as per induction hypothesis):
Every element of $Q^*_\alpha$ is a Borel function (which already exists in $V$)
applied to the generics at a countable sequence of indices in $w_\alpha$ (which also already exists in $V$).

This implies:
\begin{quote}
In $V$, $\mylinIII(\Pa,\kappa)$ holds for all $\kappa\in[\lambda_3,\lambda_5]$, witnessed by the first $\kappa$ many Cohen reals.
\end{quote}

Proof: 
Let $y$ be a $\Pa$-name of a real.
In $V^*$, 
we can interpret $y$ as $\PaB$-name, and as 
$\mylinIII(\PaB,\kappa)$ holds, we get
$(\exists \alpha\in\kappa)\, (\forall \beta\in\kappa\setminus \alpha)
\PaB\forces  c_\beta \nleq^* y$, where $c_\beta$ is the Cohen added at $\beta$.
As $\chi<\kappa$, 
there is in $V$ an upper bound $\alpha^*<\kappa$ for the possible values of $\alpha$.
For any $\beta\in\kappa\setminus \alpha^*$,
we have (in $V$) $\Pa\forces  c_\beta \nleq^* y$ (by absoluteness).
\end{proof}

To summarize: 
\begin{theorem}\label{thm:Pa}
Assuming GCH and given $\lambda_i$ as in Assumption~\ref{asm:P},
we can find parameters\footnote{I.e., we 
set $\delta_5=\lambda_5+\lambda_5$,
and find $(S^i)_{i=1,\dots,4}$ and $\bar w=(w_\alpha)_{\alpha\in \delta_5}$.}
such that the FS ccc iteration $\Pa$ as defined in~\ref{def:Pa}  satisfies, for $i=1,2,3,4$:
\begin{itemize} 
\item  
$\mylini(\Pa,\kappa)$ holds for any regular cardinal $\kappa$ in $[\lambda_i,\lambda_5]$.
\item
$\myparti(\Pa,\lambda_i,\lambda_5)$ holds.
\end{itemize}
So in particular $\Pa$ forces $\addnull=\lambda_1$, $\covnull=\lambda_2$, 
$\mathfrak{b}=\lambda_3$, $\nonmeager=\lambda_4$
and
$\covmeager=\mathfrak{d}=\nonnull=\cofnull=\lambda_5=2^{\aleph_0}$.
\end{theorem}
For the rest of the paper we fix these parameters
and thus the  forcing $\Pa$.


\section{Boolean ultrapowers}\label{sec:partB}

In Subsections~\ref{sec21} and~~\ref{sec22} we describe how to get an elementary
embedding (which we call a \emph{BUP-embedding}) $j:V\to M$ with $\crit(j)=\kappa$ and $\cf(j(\kappa))=|j(\kappa)|=\theta$,
assuming $\kappa$ is strongly compact and $\theta>\kappa$ is a regular cardinal with $\theta^\kappa=\theta$.

In Subsections~\ref{sec23} and~\ref{sec24} we show how to use such embeddings to 
transform a ccc forcing $P$ to $j(P)$ while preserving some of the
values forced to the entries of Cicho\'n's diagram (and changing others).

\subsection{Boolean ultrapowers}\label{sec21}

Boolean ultrapowers generalize ordinary ultrapowers by using arbitrary Boolean algebras
instead of the power set algebra.

We assume that  $\kappa$ is strongly compact
and that $B$ is a $\kappa$-distributive,  $\kappa^+$-cc,
atomless complete Boolean algebra.
Then 
every $\kappa$-complete filter in $B$ can be extended to a $\kappa$-complete ultrafilter $U$.\footnote{For this, neither $\kappa^+$-cc nor atomless is  required,
and $\kappa$-complete is sufficient. 
The proof is straightforward; the first proof that we are aware of has been published in~\cite{MR0166107}.} 
Also, there is a 
maximal antichain $A_0$ in $B$ of size $\kappa$
such that $A_0\cap U=\emptyset$ (i.e., $U$ is not $\kappa^+$-complete).\footnote{Proof: Let $A$ be a maximal antichain in the open dense set 
$B\setminus U$; by $\kappa^+$-cc $|A|\le \kappa$. And $A$ cannot have size ${<}\kappa$, as otherwise it would meet the $\kappa$-complete $U$.}

For now, fix some $\kappa$-complete ultrafilter $U$.

The Boolean algebra $B$ can be used as forcing notion. 
As usual, $V$ (or: the ground model) denotes the universe we ``start with''. 
In the following, we will not actually force with $B$ 
(and in this subsection and the following subsection,
we will not force with anything, we always remain in $V$);
but we still use forcing notation.
In particular, we call the usual $B$-names ``forcing names''. 

A \textdef{BUP-name} (or: labeled antichain) $x$
is a function $A\to V$ whose domain is a
maximal antichain of $B$.
We may write $A(x)$ to denote $A$.

Each BUP-name corresponds to a forcing name\footnote{more specifically, to the forcing name $\{(\widecheck{x(a)},a):\, a\in A(x)\}$.} for an element of $V$. We will identify the BUP-name 
and the corresponding forcing name. 
In turn, every forcing name $\tau$
for an element of $V$ has a forcing-equivalent BUP-name.
In particular there is a \textdef{standard BUP-name}
$\check v$ for each $v\in V$.

We can calculate,
for two BUP-names $x$ and $y$, 
the  Boolean value $ \lBrack x=y\rBrack$.
We call $x$ and $y$ \textdef{equivalent},
if $ \lBrack x=y\rBrack\in U$ (the $\kappa$-complete ultrafilter fixed above).

For example, any two standard BUP-names for the same $v\in V$ trivially are equivalent (as $\mathbb 1_B\in U$).
So we can speak (modulo equivalence) of \emph{the}  standard BUP-name for $v$.

The \textdef{Boolean ultrapower} $M^-$ 
consists of the equivalence classes $[x]$ of BUP-names $x$;
and we define $[x]\in^- [y]$ by $\lBrack x\in y\rBrack\in U$.
We are interested in the $\in$-structure
 $(M^-,\in^-)$.
We let $j^-:V\to M^-$ map $v$ to $[\check v]$.

Given  BUP-names $x_1,\dots,x_n$ and an $\in$-formula $\varphi$,
the truth value
$\lBrack \varphi^V(x_1,\dots,x_n)\rBrack$ is well defined
(it is the weakest element of $B$ forcing that in the ground model $\varphi(x_1,\dots,x_n)$ holds,
which makes sense as $x_1,\dots,x_n$ are guaranteed to be in the ground model).

A straightforward induction (which can be found in~\cite[Sec.~2]{ourselves}) shows:
\begin{itemize}
\item {\L}o{\'{s}}'s theorem: $(M^-,\in^-)\models \varphi([x_1],\dots,[x_n])$ iff
$\lBrack \varphi^V(x_1,\dots,x_n)\rBrack \in U$.
\item
$j^-:(V,\in)\to (M^-,\in^-)$ is an elementary embedding.
\item In particular, $(M^-,\in^-)$ is a  ZFC model.

\end{itemize}
%

As $U$ is $\sigma$-complete, $(M^-,\in^-)$ is wellfounded.
So we let $M$ be the transitive collapse of $(M^-,\in^-)$, and 
let $j:V\to M$ be the composition of $j^-$ with the collapse.
We denote the collapse of $[x]$ by $\BUP{x}$.
So in particular $\BUP{\check v}=j(v)$.

\begin{facts}
\begin{itemize}
\item $M\models \varphi(\BUP{x_1},\dots,\BUP{x_n})$ iff $\lBrack \varphi^V(x_1,\dots,x_n)\rBrack \in U$. In particular, $j:V\to M$ is an elementary embedding.
\item 
If $|Y|<\kappa$, then $j(Y)=j''Y$.
In particular, $j$ restricted to $\kappa$ is the identity.
$M$ is closed under ${<}\kappa$-sequences.
\item
$j(\kappa)\neq \kappa$, i.e., $\kappa=\crit(j)$.
\end{itemize}
\end{facts}
%

As we have already mentioned, an arbitrary forcing name for an element of $V$
has a forcing-equivalent BUP-name, i.e., a maximal antichain labeled
with elements of $V$.
If $\tau$ is a forcing name for an element of $Y$ ($Y\in V$),
then without loss of generality $\tau$ corresponds to a maximal antichain
labeled with elements of $Y$. We call such an object $y$ a ``BUP-name for an element
of $j(Y)$'' (and not ``for an element of $Y$'', for the obvious reason: unlike
in the case of a forcing extension, $\BUP{y}$ is generally not in $Y$,
but, by definition of $\in^-$, it is in $j(Y)$).

\begin{lemma}
If the partial order $(S,\le)$ is ${\le}\kappa$-directed, then $j''S$ is cofinal in $j(S)$.
\end{lemma}

\begin{proof}
Let $\BUP{x}$ be some element of $j(S)$; without loss of generality we can
assume that $x$ is a labeled antichain which only uses elements of $S$ as labels.
The size of the antichain is at most $\kappa$, so all labels have some common upper bound
$s_0$. Then $\lBrack x\le s_0\rBrack$ is $\mathbb{1}_B$, and thus in $U$; so
$(M^-,\in^-)\models [x]\le \check{s_0}$,
i.e., $j(s_0)\ge \BUP{x}$ as required.
\end{proof}

For later reference, let us summarize what we know about $j$ in the form of a definition:
\begin{definition}
A BUP-embedding is an elementary embedding
$j:V\to M$ ($M$ transitive) with critical point $\kappa$,
such that $M$ is ${<}\kappa$-closed and such that
$j''S$ is cofinal in $j(S)$ for every ${\le}\kappa$-directed partial order $S$.
\end{definition}
So the embedding $j$ defined as above for 
a $\kappa$-distributive, $\kappa^+$-cc atomless complete Boolean algebra
and a $\kappa$-complete ultrafilter $U$ is a BUP-embedding.

\begin{lemma}
Let $j$ be a BUP-embedding with $\crit(j)=\kappa$.
\begin{itemize}
\item\label{item:c2}
If $|A|<\kappa$, then $j''A=j(A)$.
\item
If $S$ is a ${<}\lambda$-directed partial order for some regular $\lambda<\kappa$,
then $j(S)$ is ${<}\lambda$-directed.
\item\label{item:cofinalities}
If $\cf(\alpha)\neq \kappa$ , then
$j''\alpha$ is cofinal in $j(\alpha)$, so in particular
$\cf(j(\alpha))=\cf(\alpha)$.    
\end{itemize}
\end{lemma}

\begin{proof}
For the second item, use that $M$ believes that $j(S)$ is ${<}\lambda$-directed 
and that $M$ is ${<}\kappa$-closed.
For the last item, assume $\cf(\alpha)=\lambda\ne \kappa$, witnessed by 
some strictly increasing cofinal function $f:\lambda\to \alpha$.
If $\lambda<\kappa$, then 
$M$ thinks that $j(f)$ is strictly increasing cofinal from $j(\lambda)=\lambda$
to $j(\alpha)$, which is absolute. If $\lambda>\kappa$, 
then $\alpha$ is a ${\le}\kappa$-directed (linear) order,
so $j''\alpha$ is cofinal in $j(\alpha)$.
So $j''f$, i.e.\ $(j(\zeta), j(f(\zeta)))_{\zeta\in \lambda}$, 
witnesses that $\cf(j''\lambda)=\cf(j''\alpha)=\cf(j(\alpha))$, and 
$\cf(j''\lambda)=\cf(\lambda)=\lambda$ (as these orders are isomorphic).
\end{proof}
\subsection{The algebra and the filter}\label{sec22}

For a strongly compact cardinal we can get large $\cf(j(\kappa))$:

\begin{lemma}\label{lem:embed}
Let $\kappa$ be strongly compact,  $\theta>\kappa$ and $\cf(\theta)>\kappa$.
Then there is a BUP-embedding $j$ with $\crit(j)=\kappa$ such that 
\begin{enumerate}
\item $\cf(j(\kappa))=\cf(\theta)$ and $j(\kappa)\ge \theta$.
\item $|j(\mu)|\le \max(\mu,\theta)^\kappa$ for any $\mu$.
\item\label{item:baal666} In particular, if $\theta^\kappa=\theta$ and $\kappa\le\mu\le\theta$ then $|j(\mu)|=\theta$.
\end{enumerate}
\end{lemma}

We will use this in the following form:

\begin{definition}
A ``BUP-embedding from $\kappa$ to $\theta$'' is a BUP-embedding $j$
with critical point $\kappa$ such that $\cf(j(\kappa))=|j(\kappa)|=\theta$ (in particular $\kappa$
and $\theta$ are regular).
\end{definition}

The lemma immediately implies:

\begin{corollary}\label{cor:embed}
Assume $\kappa$ is strongly compact and $\theta>\kappa$ is a 
regular cardinal such that 
$\theta^{\kappa}=\theta$. Then there is a BUP-embedding $j$ from
$\kappa$ to $\theta$. (And $|j(\mu)|=\theta$ whenever $\kappa\le\mu\le\theta$.)
\end{corollary}

\begin{proof}[Proof of Lemma~\ref{lem:embed}]
Let $B$
be the complete Boolean algebra generated by the forcing notion $P_{\kappa,\theta}$
consisting of partial functions from $\theta$ to $\kappa$ with 
domain of size ${<}\kappa$, ordered by extension.
Clearly $B$ is ${<}\kappa$-distributive 
(as $P_{\kappa,\theta}$ is even ${<}\kappa$-closed) and $\kappa^+$-cc.

%
The forcing adds a canonical generic function $f^*:\theta\to\kappa$.
So for each $\delta\in\theta$, 
$f^*(\delta)$ is a forcing name for an element of $\kappa$, and thus a BUP-name
for an element of $j(\kappa)$.

Let $x$ be some other BUP-name for an element of $j(\kappa)$,
i.e., an antichain $A$ of size $\kappa$
labeled  with elements of $\kappa$. 
As $P_{\kappa, \theta}$ is dense in $B\setminus\{\mathbb{0}_B\}$,
we can assume that $A\subseteq P_{\kappa, \theta}$. 
Let $\delta\in\theta$ be bigger than
the supremum of the domain of $a$
for each $a\in A$.
We call such a pair $(x,\delta)$ ``suitable'', and set
$b_{x,\delta} \coloneq \lBrack  f^*(\delta)>x\rBrack $.
We claim that these elements generate a $\kappa$-complete filter.
To see this, fix suitable pairs
$(x_i,\delta_i)$ for $i<\mu<\kappa$; we have to show that $\bigwedge_{i\in\mu} b_{x_i,\delta_i}\neq \mathbb{0}$. Enumerate $\{\delta_i:\, i\in \mu\}$ increasing (and without repetitions)
as $\delta^\ell$ for $\ell\in\gamma\le \mu$. 
Set $A_\ell=\{i:\, \delta_i=\delta^\ell\}$.
Given $q_\ell$,  
define $q_{\ell+1}\in P_{\kappa,\theta}$ as follows:
$q_{\ell+1}\le q_\ell$; $\delta^\ell\in \supp(q_{\ell+1})\subseteq \delta^\ell\cup\{\delta^\ell\}$; and
$q_{\ell+1}\restriction \delta^\ell$ 
decides for all $i\in A_\ell$ 
the values of $x_i$ to be some $\alpha_i$; and 
$q_{\ell+1}(\delta^\ell)=\sup_{i\in A_\ell}(\alpha_i)+1$.
This ensures that $q_{\ell+1}$ is stronger than $b_{x_i,\delta_i}$ for $i\in A_\ell$.
For $\ell\le\gamma$ limit, let $q_\ell$ be the union of $\{ q_k:\, k<\ell\}$. Then $q_{\gamma}$ is stronger than each $b_{x_i,\delta_i}$.

As $\kappa$ is strongly compact, we can extend the $\kappa$-complete filter generated by all $b_{x_i,\delta_i}$
to a $\kappa$-complete ultrafilter $U$.
Then the sequence $\BUP{f^*(\delta)}_{\delta\in \theta}$ 
is strictly increasing (as $(f^*(\delta),\delta')$ is suitable
for all $\delta<\delta'$) and cofinal in $j(\kappa)$ (as we have just seen); 
so $\cf(j(\kappa))=\cf(\theta)$ and $j(\kappa)\ge \theta$. 

To get an upper bound for $j(\mu)$ for any cardinal $\mu$, we count all possible BUP-names for elements of $j(\mu)$. As we can assume
that the antichains are subsets of $P_{\kappa,\theta}$, which has size
$\theta^{<\kappa}$, 
we get the upper bound $|j(\mu)|\le [\theta^{<\kappa}]^{\kappa}\times \mu^{\kappa}=\max(\theta,\mu)^\kappa$.
\end{proof}

\subsection{The ultrapower of a forcing notion}\label{ss:forcingBUP}\label{sec23}

We now investigate the relation of a forcing notion $P\in V$
and its image $j(P)\in M$, which we use 
as forcing notion over $V$. (Think of $P$ 
as being one of the forcings of Section~\ref{sec:partA}; it has no relation with the Boolean algebra $B$ used to construct $j$.)

Note that as $j(P)\in M$ and $M$ is transitive, every $j(P)$-generic filter $G$ over
$V$ is trivially generic over $M$ as well, and we will use absoluteness between $M[G]$ and $V[G]$ to prove various properties of $j(P)$.

\begin{lemma}\label{lem:aa}
Let $j:V\to M$ be elementary,
$M$ transitive and ${<}\kappa$-closed with $\crit(j)=\kappa$.
Assume that $P$ is $\nu$-cc for some $\nu<\kappa$. 
\begin{enumerate}
    \item\label{item:aa1} $j(P)$ is $\nu$-cc.
    \item\label{item:aa2}  If $\tau$ is (in $V$) a $j(P)$-name for an element of $M[G]$,
    then there is a $j(P)$-name $\sigma$ in $M$ such that the empty condition
    forces $\sigma=\tau$.
    \item\label{item:aa2b}
    In particular, every $j(P)$-name for a real, a Borel-code, 
    a countable sequence of reals, etc., is in $M$ (more formally: 
    has an equivalent name in $M$).
    \item\label{item:aa3}  $M[G]$ is ${<}\kappa$-closed in $V[G]$.
    \item\label{item:size} If $\xi<\kappa$ and $P$ forces $2^\xi=\lambda$, then
    $j(P)$ forces $2^\xi=|j(\lambda)|$.
    \item\label{item:aa4} 
    $j''P$, which is isomorphic to $P$ via $j$, 
    is a complete subforcing of $j(P)$. 
\end{enumerate}
\end{lemma}
\begin{proof}
(\ref{item:aa1}): If $A\subseteq j(P)$ has size $\nu$, then $A\in M$,
and by elementarity $M$ thinks that $A$ is not an antichain, which is absolute.

(\ref{item:aa2}): $\tau$ corresponds to $(A,f)$ where $A\subseteq j(P)$ is a maximal
antichain and $f:A\to M$ maps $a$ to a $j(P)$-name in $M$.
As $j(P)$ is $\nu$-cc and $M$ ${<}\kappa$-closed, $(A,f)$ is in $M$ and we can interpret
in $M$ $(A,f)$ as a $j(P)$-name $\sigma$.

This immediately implies (\ref{item:aa2b})
and
(\ref{item:aa3}): 
Given a $j(P)$-name $\tau$ for 
a $\zeta$-sequence of elements of $M[G]$, $\zeta<\kappa$,
we can interpret $\tau$ as a $\zeta$-sequence 
of names $(\tau_i)_{i<\zeta}$, and find for each 
$\tau_i$ an equivalent $j(P)$-name $\sigma_i$ in $M$.
As $M$ is ${<}\kappa$-closed, the sequence $(\sigma_i)_{i<\zeta}$ is in $M$ and defines a $j(P)$-name in $M$ equivalent to $\tau$.

(And if $\tau$ is a $j(P)$-name for a ${<}\kappa$-sequence
in $M[G]$, we can use the fact that $\kappa$ is regular and that 
$j(P)$ is $\kappa$-cc to get 
a bound $\zeta<\kappa$ for the length of $\tau$.)

(\ref{item:size})
$M[G]$ thinks that $|2^{\xi}|=j(\lambda)$,
and $2^{\xi}\cap V[G]=2^{\xi}\cap M[G]$.

(\ref{item:aa4}): It is clear that $j''P$ is an incompatibility-preserving 
subforcing of $j(P)$: 
$j(p)\le j(q)$ in $j''P$ iff 
$p\le q$ in $P$ (by definition) 
iff $M$ thinks that $j(p)\le j(q)$ in $j(P)$ (by elementarity)
iff this holds in $V$ (by absoluteness); and the same argument works for compatibility
instead of $\le$.
Similarly, assume 
$A\subseteq j''P$ is a maximal antichain. By definition,
$B\coloneq j^{-1}(A)\subseteq P$ is one as well, and in particular of size ${<}\nu$.
Therefore $j(B)=B$, and by elementarity $M$ thinks that $B\subseteq j(P)$ is maximal,
which holds in $V$ by absoluteness.
\end{proof}

To round off the picture, let us mention the following fact (which is however,
not required for the rest of the paper):
\begin{lemma}\label{lem:lemwhichisnotneededeverorisit}
If $P=(P_\alpha,Q_\alpha)_{\alpha<\delta}$ is a finite support (FS) ccc iteration of length $\delta$, then 
$j(P)$ is a FS ccc iteration of length $j(\delta)$ (more formally: it is canonically equivalent to one).
\end{lemma}

\begin{proof}
$M$ certainly thinks that $j(P)=(P^*_\alpha,Q^*_\alpha)_{\alpha<j(\delta)}$ is a
FS iteration of length $j(\delta)$.

By induction on $\alpha$ we define the 
FS ccc iteration $(\tilde P_\alpha,\tilde Q_\alpha)_{\alpha<j(\delta)}$ and show that 
$P^*_\alpha$ is a dense subforcing of  $\tilde P_\alpha$:
Assume this is already the case for $P^*_\alpha$.
$M$ thinks that $Q^*_\alpha$ is a $P^*_\alpha$-name,
so we can interpret it as $\tilde P_\alpha$-name
and use it as $\tilde Q_\alpha$.
Assume that $(p,q)$ is an element (in $V$) of $\tilde P_\alpha*\tilde Q_\alpha$.
So $p$ forces that $q$ is a name in $M$; we can strengthen $p$
to some $p'$ that decides $q$ to be the name $q'\in M$. 
By induction we can further strengthen $p'$ to $p''\in P^*_\alpha$,
then $(p'',q')\in P^*_{\alpha+1}$ is stronger than $(p,q)$.
(At limits there is nothing to do, as we use FS iterations.)

$j(P)$ is ccc according to Lemma~\ref{lem:aa}(\ref{item:aa1}).
\end{proof}

\subsection{Preservation of values of characteristics}\label{sec24}

Recall Definition~\ref{def:linear} of $\mylini$;
and Definitions~\ref{def:partial} and \ref{def:partial2} of $\myparti$.
\begin{lemma}\label{lem:bbb}
Assume\footnote{For most of the Lemma,
the requirements of Lemma~\ref{lem:aa} are sufficient: 
We use ccc only to simplify notation as we do not have to indicate 
where we calculate cofinalities (in $V$ or the 
$j(P)$ extensions $V[G]$); 
and we need BUP-embedding for 
 the last part of (2) only.}
that
$P$ is ccc
and that $j$ is a BUP-embedding with critical point $\kappa$.
\begin{itemize}
\item[(1)]\label{item:linerapreserved} $\mylini(P,\delta)$ implies  $\mylini(j(P),j(\delta))$.
    \\
    So if 
$\lambda\neq \kappa$ regular, then $\mylini(P,\lambda)$ implies  $\mylini(j(P),\lambda)$.
\item[(2)]\label{item:partialpreserved}
Assume $\myparti(P,\lambda,\mu)$. If $\kappa> \lambda$,
then $\myparti(j(P),\lambda,|j(\mu)|)$;
if 
$\kappa< \lambda$,
then $\myparti(j(P),\lambda,\mu)$.
\end{itemize}
\end{lemma}

\begin{proof}
(1)
Let $\bar x=(x_\alpha)_{\alpha<\delta}$ be the sequence of $P$-names
witnessing $\mylini(P,\delta)$.
So $M$ thinks:
For every $j(P)$-name $y$ of a real 
$(\exists \alpha\in j(\delta))\, (\forall \beta \in j(\delta)\setminus \alpha)\, \lnot\,\bigl( (j(\bar x))_\beta \Ri y\bigr)$.
This is absolute, so $j(\bar x)$ witnesses $\mylini(j(P),j(\delta))$.

The second claim follows from the fact that $\mylini(j(P),j(\delta))$ is equivalent to\linebreak $\mylini(j(P),\cf(j(\delta)))$ and that $\cf(j(\lambda))=\lambda$ for regular $\lambda\neq \kappa$. 

(2)
Let $(S,\prec)$ and $\bar g$ witness $\myparti(P,\lambda,\mu)$.
$M$ thinks that
\begin{equation}\tag{$*$}\label{eq:ff}
\text{for each $j(P)$-name $f$: }
(\exists s\in j(S))\,(\forall t\in j(S))\ (\, t\succ s\rightarrow  j(P)\forces 
f \Ri j(\bar{g})_t\,)
\end{equation}
(or, in the case $i=2$, $j(P)\forces 
j(\bar{g})_t\notin f$, where $f$ is the name of a null set).
This is true in $V$ as well: If $f$ is a $j(P)$-name for a real, then we can assume 
$f\in M$, and so we can find $s\in j(S)$ such that for all $t\succ s$,
$M[G]\models f \Ri j(\bar g)_t$, which holds in $V[G]$ as well, as $\Ri$ is absolute.

If $\lambda<\kappa$, then $j(\lambda)=\lambda$,  and $j(S)$ is $\lambda$-directed in $M$ and therefore in $V$ as well, so we get $\myparti(j(P),\lambda,|j(\mu)|)$.

So assume $\lambda>\kappa$.
We claim that $j''(S)$ and $j''\bar g$ witness $\myparti(j(P),\lambda,\mu)$.
$j''S$ is isomorphic to $S$, so directedness is trivial. 
Given a $j(P)$-name $f$ of a real, without loss of generality in $M$, there is 
in $M$ a cone with tip $s\in j(S)$ as in~\eqref{eq:ff}. As $j''S$ is cofinal in $j(S)$ 
there is some $s'\in S$ such that $j(s')\succ s$.
Then for all $t\succ s'$, i.e., $j(t)\succ j(s')$, we get 
$j(P)\forces f \Ri j(g_t)$. (Or, in case $i=2$, $j(P)\forces j(g_t)\notin f$).
\end{proof}

We list the specific cases that we will use:
\begin{corollary}\label{cor:bbb}
Let $j$ be a BUP embedding from $\kappa$ to $\theta$.
\begin{enumerate}[(a)]
\item\label{item:linerapreservedcor} $\mylini(P,\lambda)$ for a regular $\lambda\ne\kappa$ implies $\mylini(j(P),\lambda)$.
\item\label{item:linerapreservedcorb}
$\mylini(P,\kappa)$ implies $\mylini(j(P),\theta)$.
\item\label{item:partialpreservedcor}
$\myparti(P,\lambda,\mu)$ for $\kappa> \lambda$ and $\kappa\le\mu\le\theta$ implies 
$\myparti(j(P),\lambda,\theta)$.
\item\label{item:partialpreservedcord}
$\myparti(P,\lambda,\mu)$ for 
$\kappa< \lambda$ implies $\myparti(j(P),\lambda,\mu)$.
\end{enumerate}

\end{corollary}

\section{A finite iteration of BUP embeddings}\label{sec:partC}

We now have everything required for the main result:

\begin{theorem}
Assume GCH and that $\aleph_1<\kappa_9<\lambda_1<\kappa_8<\lambda_2<\kappa_7<\lambda_3<\kappa_6<\lambda_4<\lambda_5<\lambda_6<\lambda_7<\lambda_8<\lambda_9$
are regular, $\lambda_3$ a successor of a regular cardinal, 
$\lambda_i$ not successor of a cardinal with countable cofinality for $i=1,2,4,5$,
and $\kappa_i$ strongly compact for $i=6,7,8,9$.
Then there is a ccc forcing notion $\PaIX$ resulting in: 
\begin{multline*} 
\addnull=\lambda_1 < \covnull =\lambda_2 <\mathfrak{b}=\lambda_3<\nonmeager=\lambda_4<\\ <\covmeager=\lambda_5 < \mathfrak{d}=\lambda_6 < \nonnull=\lambda_7 < \cofnull=\lambda_8< 2^{\aleph_0}=\lambda_9.\end{multline*}
\end{theorem}

\begin{proof}

For $i=6,\dots,9$, let $j_i$ be a BUP-embedding from $\kappa_i$ to $\lambda_i$,
i.e., $\cf(j_i(\kappa_i))=|j_i(\lambda_i)|=\lambda_i$.
(Such an embedding exists according to Corollary~\ref{cor:embed}.)

We use $\Pa$ of Theorem~\ref{thm:Pa}, and
set  $\Paip \coloneq j_{i+1}(\Pai)$ 
for $i=5,6,7,8$.
In particular, $\PaIX = j_9(j_8(j_7(j_6(\Pa)))) $.

We enumerate the relevant characteristics of Cicho\'n's diagram as $\mathfrak x_1,\dots,\mathfrak x_8$ in the desired increasing order as displayed
in Figure~\ref{fig:ourorder}. For $i=1,\dots,4$ (i.e., $\mathfrak x_i$
in the left half) we 
set $i^* \coloneq 9-i$ (so $\mathfrak x_{i^*}$ is the dual of $\mathfrak x_i$
in the right half).

Recall that according to Lemmas~\ref{lem:linearcharacteristics} and~\ref{lem:partialcharacteristics}, 
$\mylini(\lambda)$ implies $\mathfrak x_i\le \lambda$ and 
$\mathfrak x_{i^*}\ge \lambda$; and 
$\myparti(\lambda, \mu)$ implies 
$\mathfrak x_i\ge \lambda$ and 
$\mathfrak x_{i^*}\le \mu$.

\emph{Claim:} $\PaIX$ 
forces  $2^{\aleph_0}=\lambda_9$.

\emph{Proof:} By induction on $i=5,\dots,8$ 
each $\Paip$ forces $2^{\aleph_0}=j_{i+1}(\lambda_i)=\lambda_{i+1}$
(according to Lem.~\ref{lem:aa}(\ref{item:size})
and Cor.~\ref{cor:embed}).

\emph{Claim:} $\mylini(\PaIX,\lambda_i)$ holds for $i=1,\dots,4$; as well as 
$\mylinIV(\PaIX,\lambda_5)$.

\emph{Proof:}  The statements  hold for $\Pa$ by
Thm.~\ref{thm:Pa} and are preserved by Cor.~\ref{cor:bbb}(\ref{item:linerapreservedcor}).

This implies $\mathfrak x_i\le \lambda_i$ for $i=1,\dots,4$; as well as 
$\mathfrak x_5=\covmeager\ge \lambda_5$.

\emph{Claim:} $\mylini(\PaIX,\lambda_{i^*})$ holds for $i=1,2,3$.

\emph{Proof:}  Note that $\kappa_{i^*+1}<\lambda_i<\kappa_{i*}<\lambda_5$.
So $\mylini(\Pa,\kappa_{i^*})$ holds (Thm.~\ref{thm:Pa}).
This implies $\mylini(\Paell,\kappa_{i^*})$
for $\ell=5,\dots,i^*-1$ (Cor.~\ref{cor:bbb}(\ref{item:linerapreservedcor})),
then  $\mylini(\Paell,\lambda_{i^*})$ for $\ell=i^*$
(Cor.~\ref{cor:bbb}(\ref{item:linerapreservedcorb})),
and then again 
$\mylini(\Paell,\lambda_{i^*})$ for $\ell=i^*+1,\dots,9$
(again Cor.~\ref{cor:bbb}(\ref{item:linerapreservedcor})).

This implies $\mathfrak x_{\ell}\ge \lambda_{\ell}$ for $\ell=6,7,8$.

\emph{Claim:} $\myparti(\PaIX,\lambda_i,\lambda_{i^*})$ holds for $i=1,2,3,4$.

\emph{Proof:}  $\myparti(\Pa,\lambda_i,\lambda_5)$ holds by Theorem~\ref{thm:Pa} and implies
$\myparti(\Paell,\lambda_i,\lambda_\ell)$ for $\ell=5,\dots,i^*$ (while $\kappa_\ell>\lambda_i)$ (Cor.~\ref{cor:bbb}(\ref{item:partialpreservedcor})), then
$\myparti(\Paell,\lambda_i,\lambda_{i^*})$  for $\ell=i^*+1,\dots,9$
(Cor.~\ref{cor:bbb}(\ref{item:partialpreservedcord})).

This implies $\mathfrak x_i\ge \lambda_i$ for $i=1,\dots,4$ as well as
$\mathfrak x_\ell\le \lambda_{\ell}$ for $\ell=5,\dots,8$.
\end{proof}

\section{Questions}\label{sec:questions}

The result poses some obvious questions. (Since the initial submission of the paper, some of the questions found \emph{partial answers} which we mention in the following.)

\begin{enumerate}[(a)]

\item Can we prove the result without using large cardinals?

It would be quite surprising if compact cardinals are needed,
but a proof without them will probably be a lot more complicated.

\emph{Partial answers:}
\begin{itemize}
    \item Gitik~\cite{moti2} points out that 
    certain extender embeddings are BUP-embeddings, and that
    a variation of superstrongs is sufficient to construct 
    the BUP-embeddings required in our construction.
    \item 
    As mentioned, we think that the~\emph{result} does not require
    any large cardinals. The~\emph{proof} in this paper obviously does:
    Gitik (ibid.) notes that 
    a measurable $\kappa$ with Mitchell order $\ge\kappa^{++}$ is required to get a BUP-embedding from $\kappa$ to some regular $\lambda>\kappa$. More generally,
    an easy argument given in~\cite[Sec.\ 3.1]{thirteen} (following a deeper observation pointed out by Mildenberger~\cite{MR1625907}), shows that at least $0^\#$ is required to get a constellation of models of the
    type used in our proof.\footnote{More specifically, for $j_6:V\to M$ and $G$ $j_6(\Pa)$-generic over $V$, 
    we know that a $M[G]$ is a $\kappa$-closed 
    transitive subclass of $V[G]$, cf. Lemma~\ref{lem:aa}(\ref{item:aa3}). And we have $M[G]\models \nonmeager=j_6(\lambda_4)$ and 
    $V[G]\models \nonmeager=\lambda_4<|j_6(\lambda_4)|=\lambda_6$, which implies 
    at least $0^\#$ (and probably a measurable).}
    \item
In~\cite{BCM} (building on~\cite{mejiavert}), a construction for the left half of Cicho\'n's diagram is introduced that additionally forces $\nonmeager<2^{\aleph_0}$. Accordingly, three strongly compact cardinals (or: subcompacts) are sufficient to get the ten different values.
\end{itemize}

\item Does the result still hold for other 
specific values of $\lambda_i$, such as $\lambda_i=\aleph_{i+1}$?

In our construction, 
the regular cardinals $\lambda_i$ for $i=4,\dots,9$
can be chosen quite arbitrarily (above the compact $\kappa_6$, that is).
However, $\aleph_1$, $\lambda_1,\lambda_2$ and $\lambda_3$ each have to be separated by a compact cardinal (and furthermore $\lambda_3$ has to be a successor of a regular cardinal).

\emph{Partial answer:}
In~\cite{thirteen}, it is shown that we can choose the values
quite freely. E.g., $\lambda_i=\aleph_{i+1}$ is possible;
as is basically ``any choice'' of successor cardinals.
We also show that we can replace any  number of instances of $< $  by $=$.

\item Are other linear orders between the characteristics of Cicho\'n's  diagram consistent?

\begin{figure}[b]
  \centering
 \begin{minipage}[b]{0.5\textwidth}
\[
\xymatrix@=2.5ex{
           & \lambda_3\ar[r]        & \lambda_4 \ar[r]\ar@{..}[]+<1.2em,0.5em>;[dd]+<1.2em,-0.5em>      &  \mye \ar[r]     & \lambda_8\ar[r]  &\lambda_9 \\
           &                    & \lambda_2\ar[r]\ar[u]  &  \lambda_7\ar[u] &              &\\ 
\aleph_1\ar[r] & \lambda_1\ar[r]\ar[uu] & \mye\ar[r]\ar[u] &  \lambda_5\ar[r]\ar[u]& \lambda_6\ar[uu] &
}
\]    
    \subcaption{\label{fig:alt1}An ordering compatible with FS ccc.}
  \end{minipage}%
  \begin{minipage}[b]{0.5\textwidth}
\[
\xymatrix@=2.5ex{
           & \lambda_2\ar[r]        & \lambda_7 \ar[r]      &  \mye \ar[r]     & \lambda_8\ar[r]  &\lambda_9 \\
           &                    & \lambda_4\ar[r]\ar[u]  &  \lambda_5\ar[u] &              &\\ 
\aleph_1\ar[r] & \lambda_1\ar[r]\ar[uu] & \mye\ar[r]\ar[u] &  \lambda_3\ar[r]\ar[u]& \lambda_6\ar[uu] &
}
\]    
    \subcaption{\label{fig:alt2}Another one, incompatible with FS ccc.}
  \end{minipage}
  \caption{\label{fig:alt}Alternative orderings of the cardinal characteristics.}
\end{figure}
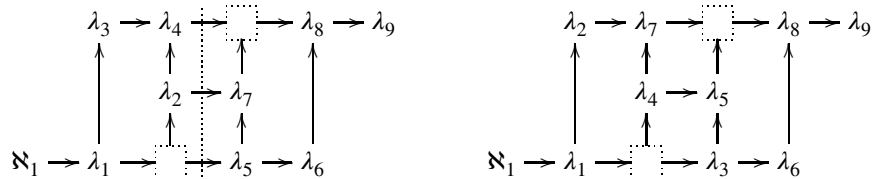

Note that in this paper, we use a FS ccc iteration of length $\delta$ with uncountable cofinality, cf.\ \ref{lem:lemwhichisnotneededeverorisit},
which always results in $\nonmeager\le\cof(\delta)\le\covmeager$.
Under these restrictions, there are only four possible
assignments.
Of course there are a lot more\footnote{In fact, we counted 57 in addition to the 4 that are compatible with FS ccc.}
possibilities to assign $\lambda_1,\dots,\lambda_8$
to Cicho\'n's  diagram in a way that satisfies the known ZFC-provable 
(in)equalities. Figure~\ref{fig:alt}.\textsc{b} is an example. 
Such orders require entirely different methods.
(Even to get just the five different values $\aleph_1=\lambda_1=\lambda_2=\lambda_3=\lambda_4=\lambda_5<\lambda_6<\lambda_7<\lambda_8<\lambda_9$ in this figure turned out to be rather involved~\cite[Sec.\ 11]{five}.)

\emph{Partial answer:}
Another of the orders compatible with FS ccc iterations,
the one of Figure~\ref{fig:alt}.\textsc{a}, is consistent~\cite{other}.
See also~\cite{modKST}.
(A different initial forcing gives the modified ordering of the left hand side; then the same construction and proof as in this paper gives us the whole diagram.) 
%

\item Is it consistent that  
other cardinal characteristics that have been studied,\footnote{The most important ones are described in~\cite{MR2768685}.} in addition to the ones
in Cicho\'n's  diagram, have pairwise different values as well? 

\emph{Partial answer:}
In~\cite{thirteen}, it is forced that additionally  $\aleph_1<\mathfrak{m}<\mathfrak{p}<\mathfrak{h}<\addnull$ holds.
\end{enumerate}


\bibliographystyle{aomalpha.bst}
\bibliography{1122}

\end{document}